\documentclass{amsart}
\usepackage[utf8]{inputenc}
\usepackage{amsmath, amssymb, amsthm, amsfonts, mathtools, array, xfrac, float, multirow, indentfirst, url, enumerate, comment, afterpage, pifont, makecell, xcolor, hyperref,comment}
\usepackage[justification=centering]{caption}
\usepackage{dsfont}
\usepackage[backend=bibtex, sorting=none, bibstyle=alphabetic, citestyle=alphabetic, sorting=nyt,maxbibnames=99]{biblatex}

\bibliography{references.bib}

\newtheorem{theorem}{Theorem}
\newtheorem{lemma}[theorem]{Lemma}
\newtheorem{proposition}[theorem]{Proposition}
\newtheorem{corollary}[theorem]{Corollary}

\theoremstyle{definition}
\newtheorem*{remark}{Remark}

\renewcommand{\Re}{\textrm{Re}}

\begin{document}

\title{A method for verifying the generalized Riemann hypothesis}

\author[G.\ Hiary, S.\ Ireland, \and M.\ Kyi]{Ghaith Hiary, Summer Ireland, Megan Kyi}
\address{
    GH: Department of Mathematics, The Ohio State University, 231 West 18th
    Ave, Columbus, OH 43210, USA
}
\email{hiary.1@osu.edu}
\address{
    SI: Department of Mathematics, The Ohio State University, 231 West 18th
    Ave, Columbus, OH 43210, USA
}
\email{ireland.118@buckeyemail.osu.edu}
\address{
    MK: Oberlin College, 135 West Lorain Street, Oberlin, OH 44074-1081, USA
}
\email{mkyi@oberlin.edu}

\subjclass[2020]{Primary: 11M06, 11Y35.}
\keywords{Riemann hypothesis, Turing test, Riemann zeta function, L-functions.}

\begin{abstract}
    Riemann numerically approximated at least three zeta zeros. According to Edwards, Riemann even took steps to verify that the lowest zero he computed was indeed the first zeta zero. This approach to verification is developed, improved, and generalized to a large class of $L$-functions. Results of numerical calculations demonstrating the efficacy of the method are presented.
\end{abstract}
%\tableofcontents
\maketitle
\section{Introduction}

Let $s=\sigma+it$ be a complex variable, where $\sigma$ and $t$ are real numbers. 
The Riemann zeta function $\zeta(s)$ is defined by the Dirichlet series
\begin{equation}\label{dirichlet series}
\zeta(s)=\sum_{n\ge 1} \frac{1}{n^s},
\end{equation}
which converges absolutely in the half-plane $\sigma >1$.
Zeta can be analytically continued to the entire complex plane except for a simple pole at $s=1$, and has zeros (i.e., roots) at $s=-2,-4,-6,\ldots$, which are called the trivial zeros. Zeta also has an infinite number of nontrivial zeros $\rho=\beta+i\gamma$ in the critical strip $0< \sigma< 1$, none of which is real. We call $|\gamma|$ the height of $\rho$ and order the $\rho$'s by increasing height.
The trivial and nontrivial zeros account for all the zeta zeros.
The Riemann Hypothesis (RH) is that all the $\rho$'s are on the critical line $\sigma=1/2$, 
or equivalently that $\beta=1/2$ for all $\rho$. 

It is frequently asserted that Riemann numerically approximated the first few zeta zeros by hand, citing unpublished notes by Riemann. See in particular Edwards~\cite[\S{7.6}]{edwards_riemann_2001} as well as the Clay Mathematics Institute page \cite{clay}.
As Figure~\ref{riemann_computation} shows, 
Riemann numerically approximated three zeta zeros on the critical line, 
corresponding to those with ordinates (i.e., imaginary parts)
\begin{equation*}
\begin{split}
    \gamma_1= 14.1347251417\ldots,\\ \gamma_2=21.0220396387\ldots,\\ \gamma_3=25.0108575801\ldots.
\end{split}
\end{equation*}
Riemann approximated $\gamma/(2\pi)$ rather than $\gamma$ as the former 
quantity appeared naturally in various formulas.

The closest approximation of $\gamma_1/(2\pi)$ that we found in Riemann's notes was $2.250466$, so that $\gamma_1$ is $\approx 14.140095$.\footnote{This differs from what is stated in \cite[159]{edwards_riemann_2001} which gave the approximation $14.1386$.} For $\gamma_2$ and $\gamma_3$, Riemann computed the approximations $3.287195$ and 
$4.0287$, respectively. Both of these approximations were 
noticeably far from the true values $3.34576152\ldots$ and $3.98060161\ldots$, and 
as can be seen in Figure~\ref{riemann_computation}, Riemann had other intermediate approximations that were 
slightly better. 
Nevertheless, Riemann's approximation of $\gamma_1$, which was
long unknown to the outside world, 
remained closest to the true value of $\gamma_1$ for nearly five decades.\footnote{As far as we can tell, the first 
circulated approximation of $\gamma_1$ was in 1887 by 
Stieltjes~\cite[450]{bailaud_bourget_1905} who gave the approximation $14.5$.
Eight years later, Gram~\cite{gram_1895}
gave the approximation $14.135$, which Gram~\cite{gram_1903} improved to $14.13472$ 
in 1903. Around the same time, Lindel\"of~\cite{lindelof_1903} devised a different method 
to approximate the $\rho$'s and
proved that $14\le \gamma_1\le 14.25$.}

According to Edwards~\cite[\S{7.6}]{edwards_riemann_2001}, Riemann even attempted 
to verify that Riemann's numerical approximation of $1/2+i\gamma_1$ indeed
corresponded to the first zeta zero (i.e., to the $\rho$ with smallest positive ordinate). 
This verification relied on the Hadamard product for $\zeta(s)$ together with
a positivity argument and  a known special value of zeta. 
However, unlike Riemann's method to numerically compute pointwise values of $\zeta(s)$, which 
became a standard method known as the Riemann--Siegel formula~\cite{siegel_1932},
Riemann's approach to  verifying the RH remained little known. 
It is worth remarking, though, 
that Gram~\cite{gram_1895,gram_1903} 
considered the power series of the logarithm of the Hadamard product, like Riemann did, 
but often appeared to assume the RH. One may also 
compare the Riemann approach to verification with the Li 
criterion \cite{li_1997} for the RH equivalence. 

After Riemann's 1859 paper, various efficient 
methods for verifying the RH were derived. 
Backlund~\cite{backlund_1911,backlund_1916} 
devised a verification method that relied on
    a clever application of 
    the argument principle from complex analysis 
 together with the Euler--Maclaurin summation for $\zeta(s)$.
    This was eventually surpassed by 
    a highly efficient method due to Turing~\cite{turing_1953}, 
   which has since become the standard method for verifying the RH, 
     \textit{provided} one is high enough on the critical line.
     
\begin{figure}[htbp]
    \includegraphics[scale = 0.3]{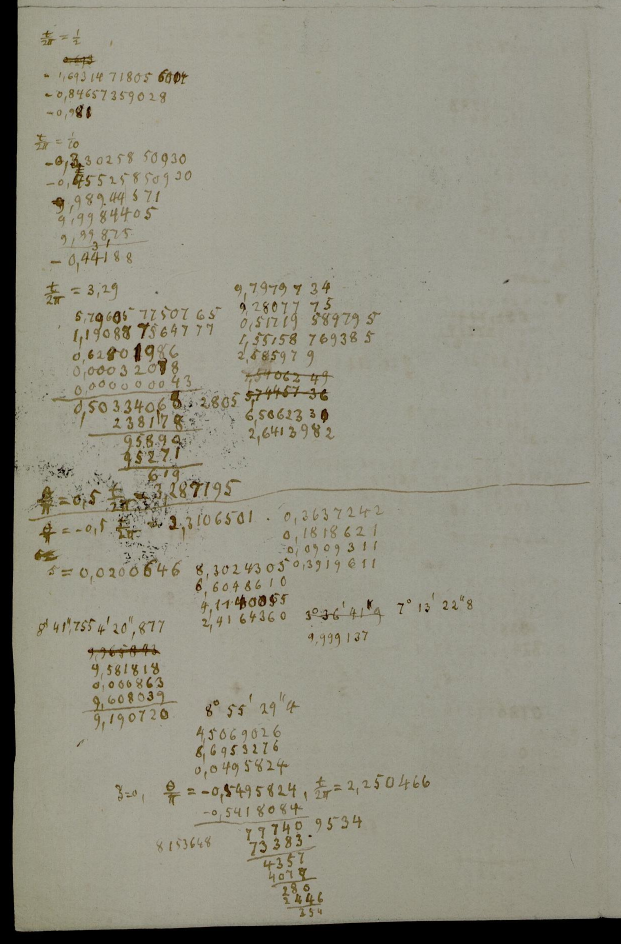}
    \includegraphics[scale = 0.3]{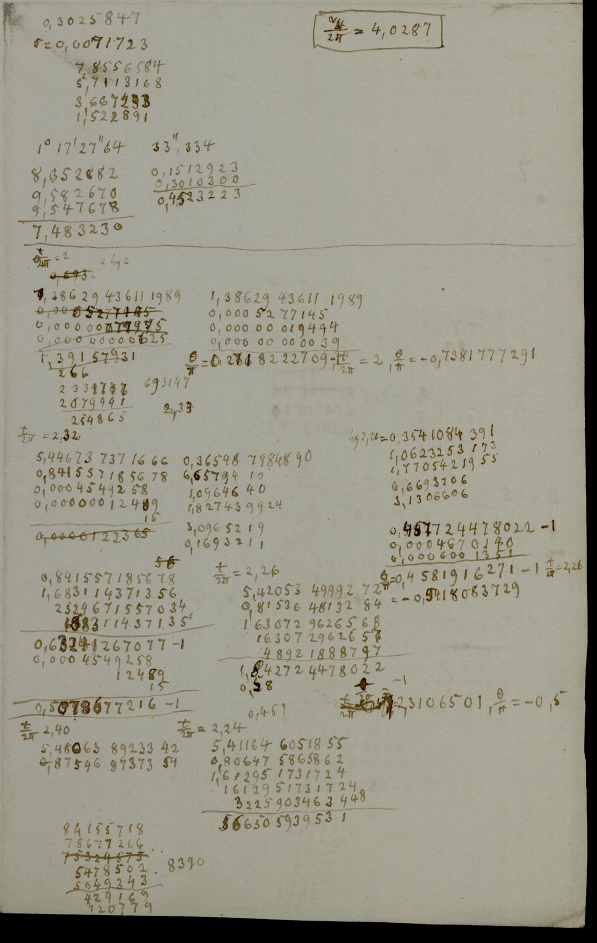}
    \includegraphics[scale = 0.3]{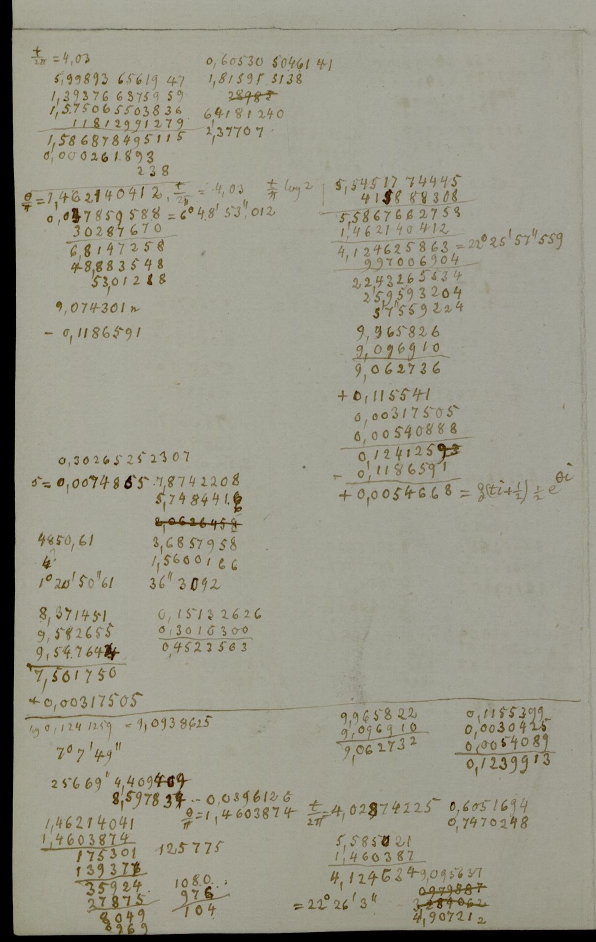}
    \caption{\small Riemann's approximation of the first three zeta zeros. Reproduced from \cite{riemann3} with permission.}
    \label{riemann_computation}
\end{figure}

In comparison, the Riemann method, cited in Edwards~\cite[\S{7.6}]{edwards_riemann_2001}, 
is time consuming at large heights. It can be expected to require
$\sim \frac{1}{2\pi^2}(t_0\log t_0)^2$ initial zeta zeros to verify the RH up to height $t_0$.
  Nevertheless, this method is reasonably efficient at low heights\footnote{For example, $10$ zeta zeros suffice to verify the RH up to height $\gamma_1$ via this method, 
and $51$ zeta zeros suffice to verify the RH up to height $\gamma_2$.} while 
  offering great simplicity.
    Therefore, it is worthwhile to generalize the Riemann method 
    to families of $L$-functions
    where even the ``first'' zero of $L(s)$ is deeply interesting. Such a generalization  is one of the main goals of this paper. 
    
    Specifically, rather 
    than fall back on a generalized Backlund method, which would require using 
    a numerically-involved application of the argument
    principle, 
    we re-examine and develop the Riemann method in a more general setting. 
    Our generalization
    works naturally with 
    already available databases and software for $L$-functions. 
    Our generalization
     is also simple to derive and justify, 
    requiring a single
    numerical evaluation of a logarithmic derivative of the $L$-function at a special point in the region of absolute convergence.

    Although the main goal of this paper is to provide a simple RH verification for a large class of $L$-functions at low heights, a secondary goal is to improve the Riemann method for zeta 
    so that it functions efficiently at large heights. This results in a conceptually straightforward verification method that can verify the RH over larger windows.  Specifically, given zeros data in a window of size $\tau$ around height $y$, the improved method in Theorem~\ref{zeta counterpart} 
 is expected to succeed in verifying the RH in a window of size $\eta\gg \tau/\sqrt{\log y}$ around $y$.

To illustrate out main results, let us state two
corollaries. 
Corollary~\ref{corollary dirichlet} provides an example of
an RH verification test for low zeros of 
 the Dirichlet $L$-function $L(s,\chi_d)$, where $\chi_d$ is any real primitive character of fundamental discriminant $d$. 
This corollary is obtained 
from Theorem~\ref{main theorem l}, part (i), on setting $\delta=-1$ and $m=1$, 
and using the formula for $w_{1,\delta}$ in Corollary~\ref{k=1 corollary l}. 

Note
the required value of the logarithmic derivative of $L(s,\chi_d)$ 
that appears in Corollary~\ref{corollary dirichlet}
is well inside the region of absolute convergence of $L(s,\chi_d)$. 
So, the required value can be computed easily by truncating the 
Dirichlet series for the logarithmic derivative, even if $d$ is very large.
Lemma~\ref{1st log deriv 1} furnishes an explicit bound on the corresponding truncation error. 

To state the next two corollaries, we will make use of the following the quantity.
$$\iota(\eta):=\min\left(\frac{1}{1+\eta^2}+\frac{2}{4+\eta^2},\, \frac{12}{9+4\eta^2}\right).$$
\begin{corollary}\label{corollary dirichlet}
Let $d$ be a positive fundamental discriminant, $\tau$ be a real positive number, 
and $\mathcal{Z}$ be a set of nonempty disjoint subintervals of the form $[\gamma_-,\gamma_+]\subseteq [0,\tau]$ or of the form $[-\gamma_0,\gamma_0]\subseteq [-\tau,\tau]$. Suppose that $L(1/2+it,\chi_d)$ has a zero of odd multiplicity in each subinterval in $\mathcal{Z}$.
Further, define  
$$ C(\mathcal{Z}):=
\sum_{[\gamma_-,\gamma_+]\in \mathcal{Z}} \frac{12}{9+4\gamma_+^2}
+ \sum_{[-\gamma_0,\gamma_0]\in \mathcal{Z}} \frac{6}{9+4\gamma_0^2}.$$
Let $\lambda_0=0.57721566\ldots$ be the Euler constant. For any real positive number $\eta$, if
$$
2\iota(\eta) +C(\mathcal{Z})
 > \frac{1}{2}\log \frac{d}{\pi e^{\lambda_0}}+ \frac{L'}{L}(2,\chi_d),
$$
then then RH holds for all the nontrivial zeros of $L(s,\chi_d)$ with positive height $\le \eta$.
\end{corollary}

    Here, one may think of $\tau$ as the width of the window where known zeros data is available, 
    and of $\eta$ as the width of the window where one would like to verify the RH. The quantity $C(\mathcal{Z})$ is the minimal contribution from known zeros (i.e.\ from supplied zeros data), 
    and $\iota(\eta)$ is the minimal contribution of a hypothetical counter-example of positive height $\le \eta$. The displayed inequality indicates that a contradiction has been reached, so that a hypothetical counter-example of positive height $\le \eta$ cannot exist. 

Similarly, by combining Theorem~\ref{main theorem zeta}, part (i) and Corollary~\ref{v1z}, 
we obtain an RH verification test for zeta zeros at large heights. 
This is stated in Corollary~\ref{corollary zeta}.
But in this case we can improve the basic verification test substantially 
  by considering the behavior 
of $S(u)$, which is the fluctuating part of the counting function of zeta zeros - see \eqref{N formula}.
Specifically, in Theorem~\ref{zeta counterpart}, we incorporate the explicit bounds 
$$|S(u)|\le \ell(u)\qquad \text{and} \qquad \left|\int_{u_0}^{u}S(\nu)\,d\nu\right|\le \ell_1(u),$$
where, according to \cite{trudgian_2014, trudgian_2011}, we may take 
\begin{align}\label{ell ell1}
    \ell(u) &:= 0.112 \log u + 0.278 \log \log u + 2.510, \\ 
    \ell_1(u) &:= 0.059 \log u+2.067, \nonumber
\end{align}
provided $u$ is large enough. ($u>u_0>168\pi$ suffices.)
We note that an explicit bound on $|\int_{u_0}^{u}S(\nu)\,d\nu|$ is a main ingredient in the Turing method as well, but we use this bound differently in our case. 
Also, without additional knowledge or analysis, the explicit bound on $|S(u)|$ is typically 
far more impactful for us than the explicit bound on $|\int_{u_0}^{u}S(\nu)\,d\nu|$.
 
We will make use of the following notation and quantity. Let
$g(s)=(s-1)\zeta(s)$, $\psi_0$ denote the digamma function, and define
$$ \kappa(y,\tau):= \frac{0.57}{\tau y^2}+\frac{3\log 2}{2\pi y}
+\frac{2\log(y/2\pi)}{\pi\tau^3}+\frac{3\log(y/2\pi)}{\pi y}
+\frac{12\, \ell(2y)}{y^2} + \frac{6\, \ell_1(2y)}{\tau^3}.$$ 

\begin{corollary}\label{corollary zeta}
Let $y$ and $\tau$ be real numbers such that $3 <\tau\le y/2$ and $336\pi < y-\tau$. 
Let $\mathcal{Z}$ be a set of nonempty disjoint subintervals of the form $[\gamma_-,\gamma_+]\subseteq [y-\tau,y+\tau]$
such that $y$ does not belong to any of the subintervals in $\mathcal{Z}$.
Suppose that $\zeta(1/2+it)$ has a zero of odd multiplicity in each subinterval in $\mathcal{Z}$.
Further, define 
$$D(\mathcal{Z}):=\sum_{\substack{[\gamma_-,\gamma_+]\in \mathcal{Z}\\ \gamma_->y}} \frac{6}{9+4(\gamma_+-y)^2}
+\sum_{\substack{[\gamma_-,\gamma_+]\in \mathcal{Z}\\ \gamma_+< y}} \frac{6}{9+4(y-\gamma_-)^2}.$$
For any real positive number $\eta\le y$, if
 \begin{align*}
      \iota(\eta) + D(\mathcal{Z}) + \frac{3}{2\pi\tau}\log\frac{y}{2\pi} - \frac{6\,\ell(2y)}{2\tau^2} & - \kappa(y,\tau) >   \\
 &  - \frac{1}{2}\log \pi +\frac{1}{2}\normalfont{\Re}\,\psi_0\left(2-\frac{iy}{2}\right)+
   \normalfont{\Re}\, \frac{g'}{g}(2-iy),
 \end{align*}
then RH holds for all the nontrivial zeros of $\zeta(s)$ with height in $[y-\eta,y+\eta]$. 
\end{corollary}

\begin{remark}
Here, one may think of $y$ as very large, $\tau$ is large but much smaller than $y$, 
and with $\eta$ somewhat smaller than $\tau$. The expression for $\kappa(y,\tau)$ 
is obtained from Theorem~\ref{zeta counterpart} by setting $x=-1$ and $c=y/2$.
\end{remark}

Like before,
the special value of the logarithmic derivative of $g(s)$ appearing in Corollary~\ref{corollary zeta}
is well inside the region of absolute convergence of $\zeta(s)$.
So this value can be approximated easily and fairly accurately 
via a truncated sum over primes and prime powers
using our Lemma~\ref{1st log deriv zeta}, even at very large heights.

Our main theorems, Theorem~\ref{main theorem l} and Theorem~\ref{main theorem zeta}, additionally
 enable verifying the simplicity of zeros in a given range as well as 
 verifying the completeness of a given list of zeros.
In addition, Theorem~\ref{zeta counterpart} gives 
a counterpart that allows
one to still draw a conclusion 
in some situations where the RH might not be verified using the
Turing method. For example, if the given zeros list is 
incomplete (i.e.\ there is a zero with ordinate in $[y-\tau,y+\tau]$ that is missing from the list), then the 
the Turing method 
might not prove that the zeros list is indeed incomplete. 
In this case, the counterpart in Theorem~\ref{zeta counterpart} 
will typically enable proving that the given zeros list is indeed incomplete.

Lastly, it completely reasonable to expect that a similar method to the one described here may be derived using the framework of the explicit formula \cite[\S{5.5}]{iwaniec_kowalski_2004}. By choosing a suitable test function in the explicit formula, one may even accelerate the convergence of the associated series over the prime and prime powers. 
At the same time though one must ensure, under no assumption, that the individual terms in 
the sum over the zeros appearing in the explicit formula are nonnegative. We favored 
the current derivation due to its simplicity, its historical connection, and because we already have good control 
over the convergence of the said series in the region of absolute convergence. 
Additionally, the current derivation gives us access to several useful exact values and exacting relations as well 
as to long-studied sums in the theory of the Riemann zeta function, which benefits the practicality of our derivation. 
For example, we can directly benefit from exact values of the polygamma function and, if we wish, 
of exact values of $L$-functions at special points such as the class number formula for Dirichlet $L$-functions. 

\vspace{1.5mm}

\noindent
\textbf{Overview}.
In \S{\ref{background}}, we provide background and set up some notation.
In \S{\ref{riemann attempt}}, we outline the Riemann approach to verifying the RH following the description in 
\cite{edwards_riemann_2001}. 
In \S{\ref{lfunction case}}, we generalize the Riemann approach to a class of $L$-functions with real Dirichlet coefficients.
In \S{\ref{zeta case}}, we treat the case of $\zeta$ separately,  
both because $\zeta(s)$ is outside our class of $L$-functions (in view of the pole at $s=1$) and 
because our focus for zeta will be on large heights.
In \S{\ref{improvements}}, we discuss substantial improvements in the case of zeta.
In \S{\ref{numerics}}, we present results of numerical computations 
implemented in interval arithmetic
for a variety of examples of $L$-functions.

\section{Background and notation}\label{background}

Using the Dirichlet series \eqref{dirichlet series}, we see that
\begin{equation}\label{real sym}
    \zeta(\overline{s})=\overline{\zeta(s)}.
\end{equation}
So, $\rho=\beta+i\gamma$ is a zeta zero if and only if the complex conjugate $\overline{\rho}=\beta-i\gamma$ is a zeta zero, or equivalently the $\rho$'s are symmetric about the real axis. 
The $\rho$'s are also symmetric about the critical line. This 
is seen by using the zeta functional equation, which 
in its simplest form states that
the entire function 
$$\xi(s):=\pi^{-s/2}\Gamma(s/2+1)(s-1)\zeta(s),$$
where $\Gamma$ is the Gamma function\footnote{The poles of $\Gamma(s/2+1)$ are all simple 
and coincide with the trivial zeros of zeta, all of which are simple as well. So, the poles of $\Gamma(s/2+1)$ cancel the trivial zeros of zeta. 
The simple pole of zeta at $s=1$ coincides with the zero of the factor $s-1$ in the definition of $\xi$.}, satisfies the functional equation 
\begin{equation}\label{func eq}
    \xi(s)=\xi(1-s).
\end{equation}
Therefore, $\xi(s)$ is even about $s=1/2$. 
For example, $\xi(0)=\xi(1) = -\zeta(0)=1/2$. 
Since $\Gamma$ and $\pi^{-s/2}$ have no zeros at all,
the zeros of $\xi$ are the same as the nontrivial zeros of $\zeta$. 
Hence, by the functional equation \eqref{func eq}, $\rho$ is a zeta zero if and only if $1-\rho$ is a zeta zero. 

Furthermore, by the functional equation \eqref{func eq} and the symmetry relation \eqref{real sym}, 
\begin{equation*}
    \xi(1/2+it)=\xi(1/2-it)=\overline{\xi(1/2+it)}.
\end{equation*}
So, $\xi$ is real-valued on the critical line (as well as on the real axis).
It follows by the intermediate value theorem that the simple (or odd multiplicity) 
nontrivial zeta zeros on the critical line correspond to sign changes of $\xi(1/2+it)$. In particular, one can numerically prove the existence of zeta zeros of odd multiplicity
on the critical line by detecting sign changes of $\xi(1/2+it)$.

\section{Riemann and verifying the RH}\label{riemann attempt}
Being an entire function of order 1, $\xi$ has 
a Hadamard product given by
\begin{equation}\label{hadamard product}
    \xi(s)=\xi(0)\prod_{\rho}(1-s/\rho),
\end{equation}
where 
the product is taken by pairing the terms for $\rho$ and $\overline{\rho}$ (or pairing the terms for $\rho$ and $1-\rho$), which ensures correct convergence. 
Starting with  \eqref{hadamard product}, Riemann obtained the following formula
\begin{equation}\label{linear case}
\sum_{\rho}\frac{1}{\rho}=v_1\qquad \text{where}\qquad 
v_1:=\frac{1}{2}\lambda_0 +1-\frac{1}{2}\log 4\pi,    
\end{equation}
and the sum over the $\rho$'s is executed by pairing the terms for $\rho$ and $\overline{\rho}$. 
Therefore,
$$v_1=2\sum_{\gamma>0}\Re\,\frac{1}{\rho}.$$
As seen in Figure~\ref{riemann_constant}, Riemann correctly computed the value of $v_1$ 
up to 20 digits,
obtaining $v_1=0.02309570896612103381\ldots$.

\begin{figure}[htbp]
%   \includegraphics[scale = 0.25]{cod3p71.pdf}
%     \vspace{1.5mm}
    \includegraphics[scale = 0.3]{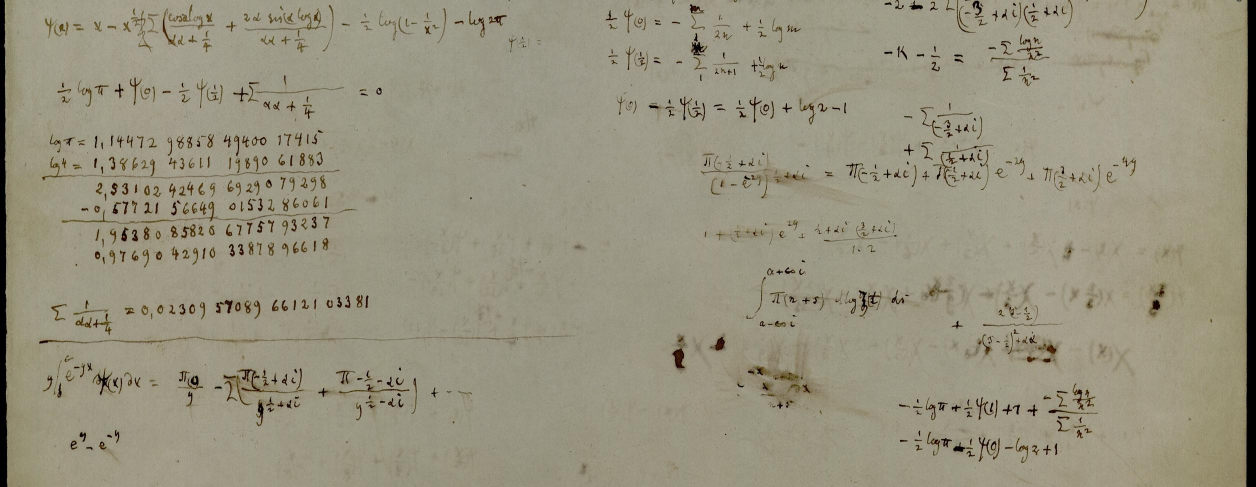}
    \caption{\small Riemann's computation of the sum over the zeros. Reproduced from \cite{riemann3} with permission.}
    \label{riemann_constant}
\end{figure}
According to Edwards~\cite[\S{7.6}]{edwards_riemann_2001}, Riemann even 
attempted to use the 
numerical value of $v_1$ to verify that the Riemann approximation of $\rho_1=1/2+i\gamma_1$ 
indeed corresponded to the first zeta zero (zeta zero of lowest height). This attempt is described essentially as follows.

Using the first $10$ zeros in the upper half-plane, $2\,\Re\,(\rho_1^{-1}+\ldots+\rho_{10}^{-1}) \approx 0.0136$.
On the other hand, 
if there is a zero $\rho_0$ in the upper half-plane of height $<\gamma_1$, 
then there must be a second such zero. This is because either 
$\rho_0$ is off the critical line, in which case $1-\overline{\rho_0}$ is a distinct zeta zero in the upper half-plane that is also of 
height $<\gamma_1$. Or $\rho_0$ is on the critical line, in which case, 
considering that $\xi(1/2+it)$  has the same sign at both $t=0$ and $t=14.1$,
there must be a second zero on the critical line with a positive ordinate $<\gamma_1$.\footnote{More precisely, 
the argument in \cite[\S{7.6}]{edwards_riemann_2001} 
only works if $\rho_0$ has height $< 14.1$. 
Since the possibility that $\rho_0$ has height $\ge 14.1$ is not yet ruled out, 
this argument  
does not force the existence of a second zero on the critical line in this case.} 
Therefore, if $\rho_0$ existed, then it would force an additional contribution of at least $2\,\Re(\rho_1^{-1})$,
 causing the zeros sum to exceed $v_1$ and hence gives a contradiction.

\begin{figure}[htbp]
%   \includegraphics[scale = 0.25]{cod3p71.pdf}
%     \vspace{1.5mm}
    \includegraphics[scale = 0.13]{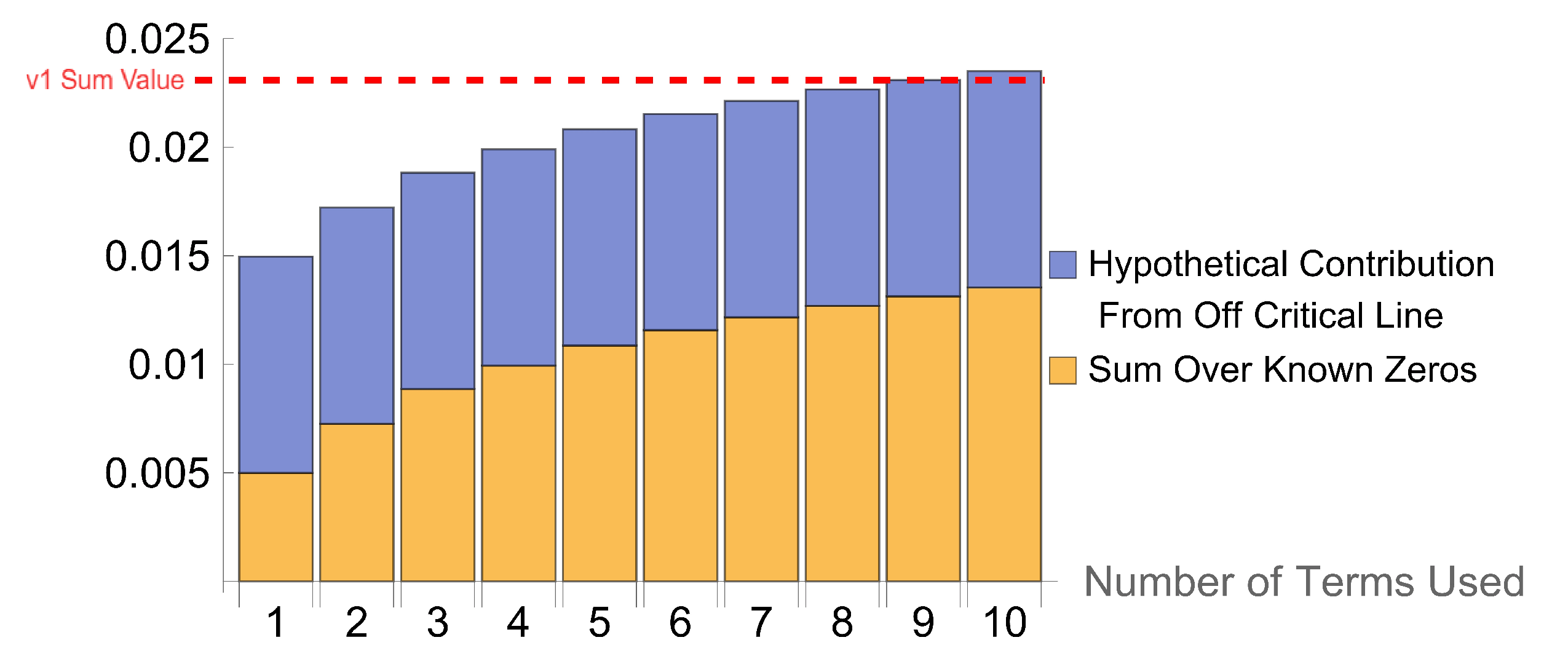}
    \caption{\small Illustration of the Riemann verification method. A contradiction is reached on using $10$ zeta zeros.}
    \label{illustration}
\end{figure}

Although not stated explicitly, it is critical to the last part of the argument that the terms 
$$\Re\,\frac{1}{\rho}= \frac{\beta}{\beta^2+\gamma^2}$$ 
are all \textit{nonnegative}. 
This ensures that the tail of the zeros sum contributes a nonnegative amount to $v_1$. 
Therefore, we can drop the tail of the zeros sum and still obtain 
a valid lower bound on $v_1$. 

More generally, in this paper, 
we will consider the behavior of the function 
\begin{equation}\label{phi def}
\phi(\beta,\eta,x):= \frac{\beta-x}{(\beta-x)^2+\eta^2}+ \frac{1-\beta-x}{(1-\beta-x)^2+\eta^2}.  
\end{equation} 
If $z=x+iy$ is a complex number then we have
$$\Re\,\left[\frac{1}{\rho-z}+\frac{1}{1-\overline{\rho}-z}\right]= 
\phi(\beta,\gamma-y,x).$$
Note that $\phi(\beta,\eta,x)$ is nonnegative for $0\le \beta\le 1$ and $x\le 0$. 
To analyze the behavior of $\phi$ in detail, we will often invoke the following lemma.

\begin{lemma}\label{phi lemma}
    Let $\beta$ be a real number such that $0\le \beta\le 1$. Let 
    $x$ be a real nonpositive number, and let $\eta$ be a real positive number. Then $\phi(\beta,\eta,x)\ge 0$. Furthermore, we have the following.
    \vspace{2mm}
      \begin{itemize}
      \setlength\itemsep{1em}
        \item[(i)] If $\displaystyle \eta \le \sqrt{\frac{x(x-1)}{3}}$, then $\phi(\beta,\eta,x)$ is minimized at $\displaystyle \beta=\frac{1}{2}$.
        \item[(ii)] If $\displaystyle \eta > \sqrt{\frac{x(x-1)}{3}}$, then $\phi(\beta,\eta,x)$ is minimized at $\beta=0$ (or $\beta=1$).
        \item[(iii)] If $\displaystyle \eta> \frac{1-2x}{2\sqrt3}$, then $\phi(\beta,\eta,x)$ is maximized at $\displaystyle \beta=\frac{1}{2}$.
        \item[(iv)] $\displaystyle\frac{\partial}{\partial u}\phi(\beta,u,x)$ is negative. Additionally, if $\displaystyle u > \frac{1-2x}{2\sqrt3}$ then $\displaystyle\frac{\partial}{\partial u}\phi(1/2,u,x)$ is increasing, and if $\displaystyle u>\frac{2-2x}{2\sqrt3}$ then $\displaystyle\frac{\partial}{\partial u}\phi(0,u,x)$ is increasing.
    \end{itemize}
    \vspace{2mm}
\end{lemma}
\begin{proof}
    See \S{\ref{lemma 6 proof}}.
\end{proof}

\section{Generalization to a class of $L$-functions}\label{lfunction case}
 
In the sequel, 
we use the analytic normalization of $L$-functions, so the critical line is $\sigma=1/2$.
We consider $L$-functions of order $1$ only.
The following notation and assumptions are used throughout this section.
Let $L(s)$ be a Dirichlet series 
$$L(s)=\sum_{n\ge 1} \frac{a(n)}{n^s},$$
absolutely convergent in the half-plane $\sigma>1$.
We suppose that the Dirichlet coefficients $a(n)$ are real, 
so that 
$$L(\overline{s}) = \overline{L(s)},$$ 
and the zeros of $L(s)$ 
must be symmetric about the real axis. 

Following the notation in Booker~\cite{booker_2006},
specialized to our context\footnote{In particular, we require that the $\mu_j$ are real and 
$\mu_j\ge 0$ instead of $\Re\,(\mu_j)\ge -1/2$. We also write the formulas 
for $\Gamma_{\mathbb{R}}(s)$ and $\overline{f}(z)$ explicitly as 
$\pi^{-s/2}\Gamma(s/2)$ and $\overline{f(\overline{z})}$, respectively, as well as
drop a scaling factor by $N^{-1/4}$ in the definition of $\gamma(s)$ in \cite[387]{booker_2006} as this
does not interfere with any of our calculations.}, we state a number of assumptions satisfied 
by the set of $L$-functions we consider.
$L(s)$ has an Euler product of degree $r$ absolutely convergent in the half-plane $\sigma > 1$,
$$L(s)=\prod_{p\,\text{prime}}\frac{1}{(1-\alpha_{p,1}p^{-s})\cdots (1-\alpha_{p,r}p^{-s})},$$
where the $\alpha_{p,j}$ satisfy the conditions in \cite[387]{booker_2006}.  We will further assume
that $|\alpha_{p,j}|\le 1$.
Note that by the absolute convergence of the Euler product, $L(s)$ has no zeros in the half-plane $\sigma>1$.

Suppose further there 
are positive integers $r$ and $N$, a complex number $\epsilon$ of modulus $1$, and real nonnegative numbers $\mu_1,\ldots,\mu_r$, such that the function $\xi_L(s)$ defined by
\begin{equation}\label{Lambda def}
    \xi_L(s) := \gamma(s)L(s),\qquad \gamma(s) :=\epsilon N^{s/2} \pi^{-sr/2}\prod_{j=1}^r \Gamma(s/2+\mu_j/2),
\end{equation}
extends to an entire function and satisfies the functional equation
\begin{equation}\label{xi L fe}
    \xi_L(s)= \overline{\xi_L(1-\overline{s})}.
\end{equation}

Note that by the functional equation, $\xi_L(1/2+it)$ is real.
Also, $\xi_L$ is real on the real axis.
If $\epsilon=\pm 1$, 
then the functional equation simplifies to $\xi_L(s)=\xi_L(1-s)$ which means that 
$\xi_L(1/2+it)$ is even in $t$. While if $\epsilon=\pm i$, then
$\xi_L(s)=-\xi_L(1-s)$ which means that $\xi_L(1/2+it)$ is odd in $t$, 
and hence must have a zero of 
odd multiplicity at $t=0$.

Since $\xi_L(s)$ is entire, $L(s)$ must have zeros at the poles of $\gamma(s)$, 
which are the trivial zeros of $L(s)$. 
Since $L(s)$ has no zeros in the half-plane $\sigma>1$,  
it follows by the functional equation that the trivial zeros of $L(s)$ in
$\sigma<0$ have the same multiplicities as the poles of $\gamma(s)$. Moreover, the
 nontrivial zeros of $L(s)$, which we 
 denote by $\rho=\beta+i\gamma$, are in the critical strip $0\le \sigma\le 1$.

We assume $L(1)\ne 0$, so that $\xi_L(1)\ne 0$, and hence $\xi_L(0)\ne 0$.
Therefore, the zeros of $\xi_L(s)$ are exactly the nontrivial zeros of $L(s)$. 
Also, just like $\xi(s)$, $\xi_L(s)$ being of order $1$ has a Hadamard product
$$\xi_L(s)=\xi_L(0)\prod_{\rho}(1-s/\rho),$$
where we pair the terms for $\rho$ and $\overline{\rho}$ (or for $\rho$ and $1-\rho$). 
 The RH for $L(s)$ is the assertion that all the $\rho$'s are on the critical line $\sigma=1/2$.

To state the next proposition, we recall the $j$-th order polygamma function $\psi_j(s)$, 
defined as the $j$-th derivative of 
$\psi_0(s) = \Gamma'(s)/\Gamma(s)$. 
Also, for any real number $\delta<1$  such that $L(1-\delta)\ne 0$, let us write 
\begin{equation}\label{log l series}
    \log L(s-\delta)=\sum_{j\ge 0} d_{j,\delta} (s-1)^j,\qquad d_{j,\delta}=\frac{1}{j!}\left[\frac{d^j}{ds^j}\log L(s-\delta)\right]_{s=1}
\end{equation}
for $s$ sufficiently close to $1$.

\begin{lemma}\label{k power l}
    Let $k$ be a positive integer. 
    Let $\delta$ be a real number such that $\delta<1$
    and $\xi_L(\delta)\ne 0$. 
    Define
    $$w_{k,\delta}:=\sum_{\rho}\frac{1}{(\rho-\delta)^k},$$
    where the sum is ordered by pairing each term with its conjugate.
    Then $w_{k,\delta}$ is a real number.
    If $k>1$, then
    $$w_{k,\delta}= (-1)^{k-1}\left[\frac{1}{2^k (k-1)!}\sum_{j=1}^r\psi_{k-1}(1/2-\delta/2+\mu_j/2)+k d_{k,\delta}\right].$$
    And if $k=1$, then the same formula holds but there is an additional term of 
    $$\frac{1}{2}\log N- \frac{r}{2}\log \pi.$$
\end{lemma}
\begin{proof}
Since $\delta$ is a real number and the $\rho$'s are symmetric about the real axis, $w_{k,\delta}$ is real.
By the Hadamard product for $\xi_L$,
\begin{equation*}
    \log \xi_L(s+\delta) = \log\xi_L(\delta)-\sum_{k\ge 1} \frac{w_{k,\delta}}{k} s^k,
\end{equation*}
provided $s$ is sufficiently close to $0$. Therefore,
\begin{equation}\label{wk formula}
    -\frac{w_{k,\delta}}{k}=\frac{1}{k!} \left[\frac{d^k}{ds^k}\log\xi_L(s+\delta)\right]_{s=0}.
\end{equation}
By the functional equation \eqref{xi L fe}, 
\begin{equation}\label{deriv fe L}
    \frac{1}{k!}\left[\frac{d^k}{ds^k}\log\xi_L(s+\delta)\right]_{s=0}=\frac{(-1)^k}{k!}\left[\frac{d^k}{ds^k}\log\xi_L(s-\delta)\right]_{s=1}.
\end{equation}
On the other hand, recalling the definition of $\xi_L(s)$,
\begin{equation*}
    \log \xi_L(s) = \log \epsilon+\frac{s}{2}\log N -\frac{sr}{2}\log \pi+\sum_{j=1}^r \log \Gamma(s/2+\mu_j/2)+\log L(s)
\end{equation*}
for $s$ away from zeros or poles of both sides. Therefore, replacing $s$ with $s-\delta$, 
and using the series expansion \eqref{log l series}, we obtain
\begin{align}\label{k deriv formula L}
    \frac{1}{k!}\left[\frac{d^k}{ds^k}\log\xi_L(s-\delta)\right]_{s=1}=&\,
    \mathds{1}_{k=1} \left(\frac{1}{2}\log N-\frac{r}{2}\log\pi\right)+ \\
    &\frac{1}{2^k k!}\sum_{j=1}^r\psi_{k-1}(1/2-\delta/2+\mu_j/2) + d_{k,\delta}, \nonumber
\end{align}
where $\mathds{1}_{k=1}$ is the indicator function of the condition $k=1$.
Substituting \eqref{k deriv formula L} back into \eqref{deriv fe L},
then back into \eqref{wk formula}, yields the proposition.
\end{proof}

Since our numerical experiments in \S{\ref{numerics}} 
will focus on the case $k=1$, we provide a version of Lemma~\ref{k power l} in this special case.
\begin{corollary}\label{k=1 corollary l} 
When $k=1$, we have
    $$ w_{1,\delta}=\frac{1}{2}\log N- \frac{r}{2}\log \pi+
    \frac{1}{2}\sum_{j=1}^r\psi_0\left(\frac{1-\delta+\mu_j}{2}\right)+
     \frac{L'(1-\delta)}{L(1-\delta)}.$$
\end{corollary}

Let us note that many special values $\psi_0(s)$ can be expressed exactly 
in terms of known constants.\footnote{For example, when $s=1/2$, 
    $\psi_0(1/2)=-2\log 2-\lambda_0$, $\psi_1(1/2)= -\pi^2/2$,
    $\psi_2(1/2) =-14\zeta(3)$, $\psi_3(1/2) = \pi^2$,
and more generally for $j\ge 1$, 
$\psi_j(1/2)= (-1)^{j+1}j!(2^{j+1}-1)\zeta(j+1)$. As another example, when $s=1$, we have $\psi_0(1)=-\lambda_0$.} 
In general, there are
efficient ways for computing $\psi_0(x)$ for $x>0$; 
see for example \cite{johansson_2021} for a discussion of methods to compute $\Gamma$, $\psi_0$, and related functions.
Therefore, for the purpose of computing $w_{1,\delta}$, we may focus our attention on the logarithmic derivative 
of $L(s)$ at $s=1-\delta$. 
The next lemma supplies a \textit{simple} formula for doing this, provided $\delta<0$. 
We make use of the following notation: 
if $n=p^m$ for a prime $p$ and a natural number $m$, then
$$ \Lambda_L(n):= \log p \sum_{j=1}^r \alpha_{j,p}^m,$$
and we set $\Lambda_L(n)=0$ otherwise. In particular, 
since $|\alpha_{j,p}|\le 1$, 
$$|\Lambda_L(n)|\le r \Lambda(n),$$ 
where $\Lambda(n)$ is the von Mangoldt function. This is defined by 
$\Lambda(n) = \log p$ if $n=p^m$ for a prime $p$ and a natural number $m$, 
and $\Lambda(n)=0$ otherwise.

\begin{lemma}\label{1st log deriv 1}
Let $K\ge 18$ be an integer. If $\delta<0$, then
$$\frac{L'(1-\delta)}{L(1-\delta)} = -\sum_{k= 1}^K \frac{\Lambda_L(k)}{k^{1-\delta}} + \mathcal{R}_L(K,\delta),$$
where
$$\left|\mathcal{R}_L(K,\delta)\right|< \frac{r K^{\delta}}{\delta} \left(2.85\cdot \frac{2\delta-1}{\log K}-1\right).$$
\end{lemma}
\begin{proof}
    Suppose $\sigma >1$. By the Euler product for $L(s)$,
\begin{equation}\label{log deriv series}
    \frac{L'}{L}(s) = -\sum_{k\ge 1} \frac{\Lambda_L(k)}{k^s}.
\end{equation}
So, by Stietljes integration and the bound $|\Lambda_L(k)|\le r \Lambda(k)$, 
the tail $\mathcal{R}_L(K,\delta)$ 
of the Dirichlet series \eqref{log deriv series} for $k>K$ and $s=1-\delta$ satisfies
\begin{equation}\label{chebychev bound 0}
    \left|\mathcal{R}_L(K,\delta)\right|\le 
     r\int_K^{\infty} \frac{1}{u^{1-\delta}}\,d\psi(u) \quad\text{where}\quad \psi(u) = \sum_{k\le u}\Lambda(k).
\end{equation}
Using integration  by parts, 
\begin{equation}\label{chebychev bound}
    \int_K^{\infty} \frac{1}{u^{1-\delta}}\,d\psi(u) = -\frac{\psi(K)}{K^{1-\delta}}+(1-\delta)\int_K^{\infty} \frac{\psi(u)}{u^{2-\delta}}\,du.
\end{equation}
Furthermore, by \cite[227]{rosser_1941}, we have for $u\ge K$ 
the double inequality 
$$0<u\left(1-\frac{2.85}{\log K}\right) \le \psi(u) \le u\left(1+\frac{2.85}{\log K}\right).$$ 
Substituting this into 
\eqref{chebychev bound}, then back into \eqref{chebychev bound 0}, and integrating yields the result.
\end{proof}

\begin{remark}
A simpler version of Lemma~\ref{1st log deriv 1} is obtained by using the trivial bound
$|\Lambda_L(k)|\le r\log k$. This gives 
\begin{equation*}\label{chebychev bound 2}
    \left|\mathcal{R}_L(K,\delta)\right|\le 
    r\sum_{n>K} \frac{\log k}{k^{1-\delta}}< r\int_K^{\infty} \frac{\log u}{u^{1-\delta}}\,du= 
    r K^{\delta}\cdot \frac{1-\delta\log K}{\delta^2}.
\end{equation*}
Although usually not as precise as Lemma~\ref{1st log deriv 1}, 
this estimate is sharper than Lemma~\ref{1st log deriv 1} 
if $\delta$ is very large compared to $\log K$.
\end{remark}

\begin{remark}
Lemma~\ref{1st log deriv 1} generalizes easily to higher order 
logarithmic derivatives of $L(s)$ at $s=1-\delta$. For example,
$$\left[\frac{d^2}{ds^s}\log L(s)\right]_{s=1-\delta} = 
-\sum_{k= 1}^K \frac{\Lambda_L(k)\log k}{k^{1-\delta}} + \mathcal{R}_{L,2}(K,\delta),$$
where
$$\left|\mathcal{R}_{L,2}(K,\delta)\right|< r K^{\delta}\left(2.85-\log K 
+\frac{1/(1-\delta)-\delta\log K}{\delta^2}(1-\delta)(1+2.85/\log K)
\right).$$
\end{remark}

Theorem~\ref{main theorem l} next is our main result in this section. Unlike the case of zeta, where none of the $\rho$'s is real, $L(s)$ might have real nontrivial zeros. So,
care is needed to allow for this possibility. 
The following lemma will facilitate the proof of Theorem~\ref{main theorem l}. Recall 
that the function $\phi$ was defined in \eqref{phi def}, and that
$$\Re\,\left[\frac{1}{\rho-\delta}+\frac{1}{1-\rho-\delta}\right]= 
\phi(\beta,\gamma,\delta).$$

\begin{theorem}\label{main theorem l}
    Let $\delta$ be a real nonpositive number and let $\tau$ be a real positive number. Let $\mathcal{Z}$ be a set of nonempty disjoint subintervals of the form $[\gamma_-,\gamma_+]\subseteq [0,\tau]$ or of the form $[-\gamma_0,\gamma_0]\subseteq [-\tau,\tau]$. Suppose that $\xi_L(1/2+it)$ has a sign change in each subinterval in $\mathcal{Z}$.\footnote{This means $\xi_L(1/2+i\tau_1)<0<\xi_L(1/2+i\tau_2)$ for some $\tau_1,\tau_2$ in each subinterval in question.}. Define 
    $$C(\mathcal{Z},\delta):=\sum_{[\gamma_-,\gamma_+]\in \mathcal{Z}} \frac{1-2\delta}{(1/2-\delta)^2+\gamma_+^2} 
    + \sum_{[-\gamma_0,\gamma_0]\in \mathcal{Z}} \frac{1/2-\delta}{(1/2-\delta)^2+\gamma_0^2}.$$
    For any real positive number $\eta$, any positive integer $m$, and with $\phi$ as in \eqref{phi def},
    define
    \begin{align*}
       f_1(\eta,\delta,m) &:=2m\cdot \min\left(\phi(0,\eta,\delta),\phi(1/2,\eta,\delta)\right), \\
       f_2(\eta,\delta,m) &:= m\cdot \phi(1/2,\eta,\delta), \\
       h_1(\delta,m) &:= m\cdot\phi(1/2,0,\delta), \\
       h_2(\delta,m) &:=m/2\cdot \phi(1/2,0,\delta), \\
       F(\eta,\delta) &:= \min\left(2\cdot\phi(0,\eta,\delta),\phi(1/2,\eta,\delta),1/2\cdot \phi(1/2,0,\delta)\right).
    \end{align*}
    Then, we have the following, where zeros are counted with multiplicity in all cases.
    \begin{itemize}
        \item[(i)] If
    $f_1(\eta,\delta,m) +C(\mathcal{Z},\delta) > w_{1,\delta}$,
    then there are strictly fewer than $4m$ non-real $\rho$'s off the critical line of height $\le \eta$.
     \item[(ii)] If
    $f_2(\eta,\delta,m) +C(\mathcal{Z},\delta) > w_{1,\delta}$,
    then there are strictly fewer than $2m$ non-real $\rho$'s on the critical line of height $\le\eta$ not accounted for in $\mathcal{Z}$.
      \item[(iii)] If $h_1(\delta,m) +C(\mathcal{Z},\delta) > w_{1,\delta}$,
    then there are strictly fewer than $2m$ real $\rho$'s off the critical line.
     \item[(iv)]If
    $h_2(\delta, m) +C(\mathcal{Z},\delta) > w_{1,\delta}$,
    then a zero at the central point $s=1/2$ has multiplicity strictly less than $m$.
    \item[(v)] If $F(\eta,\delta) +C(\mathcal{Z},\delta) > w_{1,\delta}$,
    then the list of zeros in $\mathcal{Z}$ is complete. This means that 
    every $\rho$ in the upper half-plane of height $\le \eta$ 
     is on the critical line, is simple, and belongs to some subinterval in the set $\mathcal{Z}$. 
     \end{itemize}
\end{theorem}

\begin{remark}
There are important cases where the subintervals in $\mathcal{Z}$ 
should be allowed to appear with multiplicity. In such cases, 
the conclusions about the simplicity of zeros in parts (ii) and (iv--v) should be modified 
so as to account for
any nonsimple zeros already present in $\mathcal{Z}$.
For example, the $L$-function of an elliptic curve with analytic rank $>1$ has by definition 
a zero at $s=1/2$ of multiplicity $>1$. 
Note that in this case, part (iv) of the theorem 
gives an unconditional upper bound on the analytic rank of the elliptic curve. 
Bober~\cite{bober_2013} gave a
method to bound the analytic rank of elliptic curve $L$-functions via the ``explicit formula,'' 
conditional on the RH for the corresponding $L$-function.
\end{remark}

\begin{remark}
Let us explicitly note that if $\epsilon = \pm 1$, then $\xi(1/2 + it)$ is even, so 
any zero at the central point $s=1/2$ has even multiplicity.
Thus, in this situation, it is unclear how the non-simple 
zero at $s=1/2$ can be detected rigorously by numerical means, via the intermediate value theorem, 
as there will be no sign change to detect. All this is to say that 
if $\epsilon = \pm 1$ and the zeros of height $\le \tau$ 
have been sufficiently resolved, then the sum over intervals of the form $[-\gamma_0,\gamma_0]\in \mathcal{Z}$ 
is expected to be empty.
\end{remark}

\begin{proof}
   We prove part (i). Let $\{\rho_1,\ldots, \rho_m\}$ be a set of $m$ zeros $\rho_j = \beta_j + i\gamma_j$ off of the critical line of height $\le \eta$ with $\Re(\rho_j) > \frac12$ and $\Im(\rho_j) > 0$, possibly with repetition up to multiplicity. For each $\rho_j$, there are necessarily 4 symmetric, distinct zeros $\rho_j$, $1-\rho_j$, $\overline{\rho_j}$, and $1-\overline{\rho_j}$. These 4 counterexample zeros will collectively contribute to $w_{1,\delta}$ a value of
   $$
   \phi(\beta_j, \gamma_j, \delta) + \phi(\beta_j, -\gamma_j, \delta) = 2\phi(\beta_j,\gamma_j,\delta).
   $$
   Using lemma \ref{phi lemma} and the monotonicity of $\phi(\beta,\eta,x)$ in $\eta$, we find that the minimum of the possible contribution from these 4 counterexample $\rho_j$ is at least
   $$
2\cdot \min_{\beta\in [0,1]}\phi(\beta,\eta,\delta) = 2\cdot \min(\phi(0,\eta,\delta),\phi(1/2,\eta,\delta)) = f_1(\eta,\delta,1).
   $$
   Note that this lower bound is independent of $\beta_i$ and $\gamma_i$. Thus, the $m$ counterexample zeros $\rho_1,\ldots, \rho_m$ along with their $3m$ symmetric zeros contribute at least $m\cdot f_1(\eta,\delta,1) = f_1(\eta,\delta, m)$ to the value of $w_{1,\delta}$. Thus, if
   $$f_1(\eta,\delta,m) + C(\mathcal{Z},\delta) > w_{1,\delta},$$
   then we have a contradiction, so there are strictly fewer than $4m$ zeros $\rho$ off the critical line of positive 
   height $\le \eta$. In the case of $m=1$ this means the RH holds in the interval $(0,\eta]$. In the case of $m=2$, 
   there is at most one set of 4 symmetric, non-real $\rho$ off the critical line of height $\le \eta$ and they must be simple.
   
   Next, we prove part (ii). Note that the minimum contribution to $w_{1,\delta}$ from a non-real zero $\rho$ on the critical line of height $\le \eta$ together with its symmetric part $1-\rho$ is $\phi(1/2,\eta,\delta)$. Similarly, the contribution from $m$ such zeros on the critical line is at least $m\cdot\phi(1/2,\eta,\delta) = f_2(\eta,\delta,m)$. Therefore, if
   $$f_2(\eta,\delta,m) + C(\mathcal{Z},\delta) > w_{1,\delta},$$
   then we have a contradiction, so there are strictly fewer than $2m$ zeros $\rho$ on the critical line of positive height $\le \eta$ which are not accounted for in the subintervals of $\mathcal{Z}$. In the case of $m=1$ this means the intervals in $\mathcal{Z}$ account for all non-real $\rho$ on the critical line of height $\le\eta$. In the case of $m=2$, at most one pair of non-real $\rho$ and $1-\rho$ on the critical line of height $\le\eta$ are not accounted for in $\mathcal{Z}$. Since each subinterval of $\mathcal{Z}$ contains a sign-change (so corresponds to a zero of odd multiplicity), this case implies that all non-real $\rho$ on the critical line of height $\le\eta$ are simple (including the possible pair $\rho$ and $1-\rho$ missed by $\mathcal{Z})$. 
   %For general $m$, such zeros potentially missed may occur within subintervals in the set $\mathcal{Z}$ or outside of them entirely.\\

   We prove part (iii). By Lemma~\ref{phi lemma}, the minimum of the contribution to $w_{1,\delta}$ from a real zero $\rho$ off the critical line and its symmetric zero $1-\rho$ is $\phi(1/2,0,\delta)$. Thus, the contribution from $m$ pairs of real zeros off the critical line $\{\rho_1,1-\rho_1,\ldots, \rho_m,1-\rho_m\}$ is at least $m\cdot \phi(1/2,0,\delta) = h_1(\delta,m)$. So if 
   $$h_1(\delta,m) + C(\mathcal{Z},\delta) > w_{1,\delta},$$ 
   there are strictly fewer than $m$ such pairs of real zeros, so fewer than $2m$ total real zeros off the critical line. 
   
   We prove part(iv). Each repetition of the zero $\rho=1/2$ (possibly none) 
   contributes $1/2\cdot\phi(1/2,0,\delta)$ to $w_{1,\delta}$. So, if the zero at $s=1/2$ has multiplicity $m$, then these zeros have total contribution $m/2\cdot\phi(1/2,0,\delta) = h_2(\delta,m)$. By the same arguments thus far, if
   $$h_2(\delta,m) + C(\mathcal{Z},\delta) > w_{1,\delta},$$ 
   then the multiplicity of the zero at $s=1/2$ is strictly smaller than $m$. 
   In the case $m=1$, this means we have non-vanishing of $L(s)$ on the real line. In the case $m=2$, and combined with part (iii), 
   any real zero must be at the central point $s=1/2$ and must be simple.

  Lastly, we prove part (v). Suppose 
  $F(\eta,\delta) + C(\mathcal{Z}, \delta) > w_{1,\delta}$.
  By the definition of $F(\eta,\delta)$, we have
       $$F(\eta,\delta) \le \min\{f_1(\eta,\delta,1),f_2(\eta,\delta,1),h_2(\delta,1)\}.$$
   Therefore, by part (i) of the theorem, all non-real $\rho$ are on the critical line. By part (ii), all non-real $\rho$ on the critical line belong to some subinterval in the set $\mathcal{Z}$ and are thus simple. By parts (iii) and (iv), $L(s)$ is non-vanishing on the real line (except for possibly a simple zero at $\rho=1/2$ included in $\mathcal{Z}$). These three cases leave no room for zeros outside of the simple zeros within the subintervals in the set $\mathcal{Z}$. Thus, the list of zeros in $\mathcal{Z}$ account for all zeros of $L(s)$ of height $\le\eta$ and they are all simple, i.e. the list $\mathcal{Z}$ is complete.
\end{proof}

\section{Generalization in the zeta case}\label{zeta case}

Let
$g(s):=(s-1)\zeta(s)$. So, $g$ is an entire function. 
The series expansion of $g$ at $s=1$ is given by
\begin{equation}
    g(s)=1+\sum_{j\ge 0} \frac{(-1)^j\lambda_j}{j!}(s-1)^{j+1},
\end{equation}
where $\lambda_1,\lambda_2,\ldots$ are the Stieltjes constants.\footnote{For instance, 
$\lambda_1=-0.07281584\ldots$, $\lambda_2=-0.00969036\ldots$, $\lambda_3 = 0.00205383\ldots$,
and so on.} For any complex number $z$ such that $\zeta(1-z)\ne 0$, we may write
\begin{equation*}
    \log g(s-z)=\sum_{j\ge 0} c_{j,z} (s-1)^j,\qquad c_{j,z}=\frac{1}{j!}\left[\frac{d^j}{ds^j}\log g(s-z)\right]_{s=1},
\end{equation*}
for $s$ sufficiently close to $1$.\footnote{If $z=0$, then the coefficients $c_j:=c_{j,0}$ can be calculated 
easily in terms of the $\lambda_j$'s. For example, $c_0=0,\,
        c_1=\lambda_0,\, c_2=-\lambda_0^2/2-\lambda_1,\,
        c_3=\lambda_0^3/3+\lambda_0\lambda_1+\lambda_2/2\ldots$.}

\begin{lemma}\label{k power}
Let $k$ be a positive integer. Let $z=x+iy$ be a complex number such that $x<1$ and $z$ does not 
coincide with any zero $\rho$ of $\xi(s)$. If $k>1$ then
$$v_{k,z}:=\sum_{\rho} \frac{1}{(\rho-z)^k}=(-1)^{k-1}\left[\frac{\psi_{k-1}(3/2-z/2)}{2^k(k-1)!}+k c_{k,z}\right].$$
If $k=1$, then there is an additional term of 
$$\displaystyle -\frac{1}{2}\log\pi.$$
\end{lemma}
\begin{proof}
The proof is similar to that of Proposition~\ref{k power l} 
except the coefficients $c_{j,z}$ are defined differently than the 
analogous coefficients $d_{j,\delta}$ 
due to the pole of zeta.
\end{proof}

\begin{corollary}\label{v1z}
When $k=1$, we have
    $$ v_{1,z}= - \frac{1}{2}\log \pi+
    \frac{1}{2}\psi_0\left(\frac{3-z}{2}\right)+
    \frac{g'(1-z)}{g(1-z)}.$$
\end{corollary}

One can compute the $g'(1-z)/g(1-z)$
using the Euler--Maclaurin summation formula; see for example~\cite{rubinstein_2005}. 
However, if $x<0$ is large enough, then the following simpler formula could suffice, 
and has the same proof as that for Lemma~\ref{1st log deriv 1}.

\begin{lemma}\label{1st log deriv zeta}
If $z=x+iy$ and $x< 0$, then
$$\frac{g'(1-z)}{g(1-z)} = -\frac{1}{z}-\sum_{k= 1}^K \frac{\Lambda(k)}{k^{1-z}} + \mathcal{R}(K,x),$$
where $\mathcal{R}(K,x)$ satisfies the same bound as in Lemma~\ref{1st log deriv 1} but with $r=1$ and $\delta=x$.
\end{lemma}

Theorem~\ref{main theorem zeta} is the main result in this section. Since the main interest in the case of zeta is at large heights, we expand about a complex number $z=x+iy$ where $y>0$ is typically large. Therefore, 
 the advantage provided by the symmetry of the $\rho$'s about the real axis is mostly lost. 

\begin{theorem}\label{main theorem zeta}
    Let $z=x+iy$ be a complex number such that $x\le 0$ and $y>0$. 
    Let $\tau$ be a real positive number such that 
    $\tau\le y$. Let $\mathcal{Z}$ be a set of nonempty disjoint subintervals $[\gamma_-,\gamma_+]\subseteq [y-\tau,y+\tau]$ such that $\xi(1/2+it)$ has a sign change in each subinterval. 
    Suppose further that $y$ does not belong to any of the subintervals in $\mathcal{Z}$.
     Define
     \begin{align*}
     D(\mathcal{Z},z)&:=\sum_{\substack{[\gamma_-,\gamma_+]\in \mathcal{Z}\\ \gamma_->y}} \frac{1/2-x}{(1/2-x)^2+(\gamma_+-y)^2}+\sum_{\substack{[\gamma_-,\gamma_+]\in \mathcal{Z}\\ \gamma_+< y}} \frac{1/2-x}{(1/2-x)^2+(y-\gamma_-)^2}. 
     \end{align*}
     Further, for any real $\eta$ such that $0<\eta\le y$, and with $\phi$ as in \eqref{phi def}, define
    \begin{align*}
g_1(\eta,x)&:=\min\left(\phi(0,\eta,x),\phi(1/2,\eta,x)\right), \\
        g_2(\eta,x)&:= \phi(1/2,\eta,x), \\
        g_3(\eta,x) &:=\min\left(\phi(0,\eta,x),1/2 \cdot \phi(1/2,\eta,x)\right).
    \end{align*}
    Then, for any real positive number $\eta$ such that $\eta\le y$ we have the following.
     \begin{itemize}
        \item[(i)] If
    $g_1(\eta,x) +D(\mathcal{Z},z) > \normalfont{\Re}\,(v_{1,z})$,
    then all the $\rho$'s with height in $[y-\eta,y+\eta]$ are on the critical line.
    That is, the RH holds in the interval $[y-\eta,y+\eta]$.
    \item[(ii)] If
    $g_2(\eta,x) +D(\mathcal{Z},z) > \normalfont{\Re}\,(v_{1,z})$,
    then all the $\rho$'s on the critical line with height in $[y-\eta,y+\eta]$ are simple.
    \item[(iii)] If $g_3(\eta,x) +D(\mathcal{Z},z) > \normalfont{\Re}\,(v_{1,z})$,
    then the list $\mathcal{Z}$ is complete. This means that  
    every $\rho$ in the upper half-plane with height in $[y-\eta,y+\eta]$
     is on the critical line, is simple, and belongs to some subinterval in the set $\mathcal{Z}$. 
     \end{itemize}
\end{theorem}
\begin{proof}
     Let us prove part (i). Suppose there is a counter-example $\rho=\beta+i\gamma$ such that $\gamma\in [y-\eta,y+\eta]$. Then $1-\overline{\rho}$ is a counter-example distinct from $\rho$. The contribution of $\rho$ and $1-\overline{\rho}$ to $\Re(v_{1,z})$ is $\phi(\beta,\gamma-y,x)$. Since $|\gamma-y|\le \eta$, it follows by Lemma~\ref{phi lemma} that this contribution is at least $g_1(\eta,x)$. Moreover, the zeros from the set $\mathcal{Z}$ already contribute at least $D(\mathcal{Z},z)$ to $\Re(v_{1,z})$. So, if the inequality in (i) holds, and considering that any remaining zeros will contribute a nonnegative amount to $\Re(v_{1,z})$, then we obtain a contradiction. Hence, the counter-example $\rho$ cannot exist.

     We prove part (ii). Suppose there is a nonsimple zero $\rho=1/2+i\gamma$ of multiplicity $m$ such that $\gamma\in [y-\eta,y+\eta]$. If $\rho$ is already in the set $\mathcal{Z}$, then $m\ge 3$, since the zeros in $\mathcal{Z}$ have odd multiplicity (as they correspond to sign changes of $\xi(1/2+it)$). If $\rho$ is not in $\mathcal{Z}$, then $m\ge 2$. In either case, there are at least two zeros on the critical line with ordinates in $[y-\eta,y+\eta]$ that are missing from $\mathcal{Z}$. So, arguing as in part (i) and using Lemma~\ref{phi lemma}, the contribution of these missing zeros to $\Re(v_{1,z})$ is at least $g_2(\eta,x)$. So, if the inequality in (ii) holds, then we obtain a contradiction since any remaining zeros will contribute a nonnegative amount to $\Re(v_{1,z})$. Hence, such a nonsimple  $\rho$ cannot exist.

      Lastly, we prove part (iii). Note that $g_1(\eta,x)\ge g_3(\eta,x)$ and $g_2(\eta,x)\ge g_3(\eta,x)$. So, if the inequality in (iii) holds, then all the zeros with height in $[y-\eta,y+\eta]$ are on the critical line and are simple. Thus, 
      in seeking a contradiction we may assume without loss of generality that there is a simple zero $\rho=1/2+i\gamma$ such that $\gamma\in [y-\eta,y+\eta]$ and $\gamma$ is not in any subinterval $[\gamma_-,\gamma_+]\in \mathcal{Z}$. 
      But the contribution of such $\rho$ to $\Re(v_{1,z})$ is at least $1/2 \cdot \phi(1/2,\eta,x)$. Hence, if the inequality in (iii) holds, then we obtain a contradiction, like before. So, such a missing $\rho$ cannot exist.
\end{proof}

\begin{remark}
By using the shift $z=-1/2+i 14.1$ in Theorem~\ref{main theorem zeta} along with the $12$ initial zeros of $\zeta(s)$, one can verify that $\rho_1=1/2+i\gamma_1$ and $\rho_2=1/2+i\gamma_2$ are the only zeta zeros with ordinates in the window $[6.5360, 21.6640]$. Since the value $v_1= 0.0230957\ldots$ that Riemann computed already tells us that there are no zeta zeros of height less than $6.56$, this yields that $\rho_1$ and $\rho_2$ are indeed the first two zeta zeros. By comparison, verifying $\rho_1$ and $\rho_2$ are the first two zeta zeros using just the value $v_1$ requires accounting for the contribution of $52$ initial zeros of zeta. %Thus, maximizing the performance of this algorithm involves a balance between the amount of zero data one has to work with and the ease of computing $v_{1,z}$.
\end{remark}

\section{Improvements}\label{improvements}

Instead of using nonnegativity to simply drop the contribution to $\Re(v_{1,z})$ 
of the tail of the zeros sum, 
we derive  a lower bound on the contribution of the tail.
 Incorporating this into Theorem~\ref{main theorem zeta}  greatly 
 improves the efficiency of our verification method 
 at large heights (i.e.\ when $y$ is large). Hence, 
the RH can be verified via our method 
in a much wider window than before (i.e.\ for a much larger $\eta$). Specifically, whereas the basic verification method in Theorem~\ref{main theorem zeta} is only expected to succeed 
in windows of size $\eta \ll \sqrt{\tau/\log y}$, the improved method in Theorem~\ref{zeta counterpart} 
 is expected to succeed in windows of size $\eta\gg \tau/\sqrt{\log y}$.
 
In addition, we derive an upper bound on the contribution of the tail of the zeros sum. 
This can sometimes allow us to prove the incompleteness of a supplied 
list of zeros in a given range, as shown in Theorem~\ref{zeta counterpart}.

\begin{proposition}\label{zeros sum bound}
Let $z=x+iy$ be a complex number and $\tau$ be a real number. 
Suppose that $x< 0$ and $1-2x<\tau< y$.
For any real number $c$ such that $168\pi<c<y-\tau$, we have
    $$ - b(z,\tau,c) \le \sum_{\substack{\rho\\ |\gamma-y|>\tau}}\normalfont{\Re}  \frac{1}{\rho-z}- \frac{1-2x}{2\pi \tau}\log\frac{y}{2\pi} \le B(z,\tau,c),$$
    where
    \begin{align*}
        &b(z,\tau,c):= \frac{1}{2\pi}\cdot \left[\epsilon_1\frac{1-2x}{\tau}+\epsilon_2+\epsilon_3\log \frac{y}{2\pi} \right]
+ \frac{\epsilon_4+\epsilon_5}{2},\\
 &B(z,\tau,c):= \frac{1}{2\pi}\cdot \left[\epsilon_1\frac{1-2x}{\tau}+\epsilon_2\right]
+ \frac{\epsilon_4+\epsilon_5}{2}+\epsilon_6,
    \end{align*}
and defining $\ell(u)$ and $\ell_1(u)$ as in \eqref{ell ell1} we have
\begin{align*}
    &\epsilon_1(y,\tau,c):=4\pi^2\cdot 0.006\cdot \left[\frac{1}{(y+\tau)^2}+\frac{1}{c^2}\right],\\
    &\epsilon_2(z,c):=\frac{1-2x}{2y}\cdot \log\frac{2y}{c},\\
    &\epsilon_3(z,\tau,c):= \left[\frac{(1-x)^2}{3\tau^3}+\frac{1}{y-c}\right]\cdot (1-2x),\\
    &\epsilon_4(z,\tau,c) := \left(\frac{2-4x}{\tau^2}+\frac{2-4x}{(y-c)^2}\right)\cdot \ell(2y),\\ 
    &\epsilon_5(z,\tau) := \frac{4-8x}{\tau^3} \cdot \ell_1(2y),\\
    &\epsilon_6(z,c) := \frac{1-2x}{2y}\cdot \frac{(2y-c)^2\log(2y-c)-(y-c)^2\log(y-c)}{\pi (y-c)^2}.
\end{align*}
%bounds on $|S(u)|$ and $|\int_{u_0}^{u} S(t)\,dt|$ which are given by  
% $\ell(u) := 0.112 \log u + 0.278 \log \log u + 2.510$ and  $\ell_1(u) := 0.059 \log u+2.067$.
%and $\epsilon_1,\ldots,\epsilon_5$ are defined in \eqref{eps 1 def},\eqref{eps 2 def},\eqref{eps 3 def},\eqref{eps 4 def}, \eqref{eps 5 def}.
\end{proposition}
\begin{proof}
See \S{\ref{proposition 12 proof}}.
\end{proof}

\begin{theorem}\label{zeta counterpart}
    Let $z =x+iy$, $\tau$, $c$, and the functions $b(z,\tau,c)$ and $B(z,\tau,c)$ all be given as in Proposition \ref{zeros sum bound}. Furthermore, let $\mathcal{Z}$ and the functions $D(\mathcal{Z},z)$, $\phi(\beta,\eta,x)$, $g_1(\eta,x)$, $g_2(\eta,x)$, and $g_3(\eta,x)$ be given as in Theorem \ref{main theorem zeta}. Define
    \begin{align*}
        r(z,\tau,c) &:= \frac{1-2x}{2\pi\tau}\log\frac{y}{2\pi} - b(z,\tau,c)\\
        R(z,\tau,c) &:= \frac{1-2x}{2\pi\tau}\log\frac{y}{2\pi} + B(z,\tau,c)
    \end{align*}
    For any real positive number $\eta$ such that $\eta \le y$ we have the following improvements to Theorem~\ref{main theorem zeta}.
    \begin{itemize}
        \item[(i)] If
        $g_1(\eta,x) + D(\mathcal{Z},z) + r(z,\tau,c) > \normalfont{\Re}\,(v_{1,z})$,
        then all the $\rho$'s with height in $[y-\eta,y+\eta]$ are on the critical line. That is, the RH holds in $[y-\eta,y+\eta]$.
        \item[(ii)] If
        $g_2(\eta,x) +D(\mathcal{Z},z) + r(z,\tau,c) > \normalfont{\Re}\,(v_{1,z})$,
        then all the $\rho$'s on the critical line with height in $[y-\eta,y+\eta]$ are simple.
        \item[(iii)] If $g_3(\eta,x) +D(\mathcal{Z},z) + r(z,\tau,c) > \normalfont{\Re}\,(v_{1,z})$,
        then every $\rho$ in the upper half-plane with height in $[y-\eta,y+\eta]$
        is on the critical line, simple, and belongs to some subinterval in the set $\mathcal{Z}$. 
     \end{itemize}
     In addition to these improvements, the upper bound in Proposition \ref{zeros sum bound} yields the following counterpart.
     \begin{itemize}
         \item[(iv)] If $D(\mathcal{Z},z) + R(z,\tau,c) < \normalfont{\Re}\,(v_{1,z})$, then $\mathcal{Z}$ does not account for all the $\rho$'s with height in $[y-\tau,y+\tau]$. This means there is a subinterval in $\mathcal{Z}$ that contains the ordinates of at least three $\rho$'s (including multiplicity), or there is $\rho=1/2 + i\gamma$ such that $\gamma\in [y-\tau,y+\tau]$ and $\gamma$ is not in any subinterval in $\mathcal{Z}$,  or there is $\rho$ off the critical line with height in $[y-\tau,y+\tau]$. 
    \end{itemize}
\end{theorem}
\begin{proof}
    Parts (i)--(iii) follow directly from the arguments in Theorem 10, except that these bounds account for the contribution from zeros $\rho$ outside of the ordinate window $[y-\tau,y+\tau]$ for which we have zeros data.

    For part (iv), if $D(\mathcal{Z},z) + R(z,\tau,c) < \normalfont{\Re}\,(v_{1,z})$, then there necessarily are zeros whose (positive) contribution to $\normalfont{\Re}\,(v_{1,z})$ is not being accounted for. More explicitly, since
    \[
    \sum_{\substack{\rho\\ |\gamma-y|>\tau}}\normalfont{\Re} \left(\frac{1}{\rho-z}\right) \leq R(z,\tau,c),
    \]
    $R(z,\tau,c)$ already accounts for the maximum possible contribution from all zeros $\rho$ with $|\gamma - y| > \tau$. Therefore, any deficiency in contribution to $\Re(v_{1,z})$ must arise from some $\rho = \beta + i\gamma$ satisfying $|\gamma - y|\le \tau$ 
    that has not been already accounted for 
    in $\mathcal{Z}$.
\end{proof}

\begin{remark}
    It is possible that a further small improvement would be made by incorporating explicit zeros-density estimates, in addition to the explicit bounds on $S(u)$ and its integral that are already included. 
    %It would also be reasonable to adapt the method of Theorem~\ref{zeta counterpart} 
    %for proving there are no zeros of height $[y-\eta,y+\eta]$  
    %in a narrow vertical strip of the form $\sigma_0 <\sigma <1$, where $\sigma_0>1/2$.
    %For comparison, Theorem~\ref{zeta counterpart} is currently formulated so as to prove the stricter statement that 
    %there are no zeros of height $[y-\eta,y+\eta]$ in the wider vertical strip $1/2<\sigma <1$. 
\end{remark}

\section{Numerical examples}\label{numerics}

The examples in this section are meant for illustration, to show how the method we described behaves 
in practice on representative examples.
The data in this section was obtained from \cite{lmfdb} and \cite{zeta_zeros}, 
and using LCALC~\cite{lcalc} as well as SageMath~\cite{sagemath}.
Our working assumption is that the zeros ordinates from \cite{lmfdb} and \cite{zeta_zeros} are accurate within $\pm 10^{-10}$, and the zeros ordinates obtained using \cite{lcalc} and \cite{sagemath} 
are accurate to within $\pm 10^{-8}$, though it is possible the accuracy is higher.
We used this assumption to determine the interval $[\gamma_-,\gamma_+]$ corresponding to each zero ordinate $\gamma$.
Numerical calculations were done using the interval arithmetic package in mpmath~\cite{mpmath}. We also used FLINT~\cite{flint} to compute the polygamma function when no exact value was available. The code for the implementation is available as a GitHub repository~\cite{github}.

\subsection{The Riemann zeta function} 
We used Theorem~\ref{main theorem zeta} and Theorem~\ref{zeta counterpart}
for verification using
$$y = 10^{28} + 501675.8,\qquad x=-2,\qquad \tau = 501575.4,\qquad c= y/2.$$
Our set $\mathcal{Z}$ was obtained from \cite{zeta_zeros}. For $z=x+iy$, 
and with the aid of Corollary~\ref{v1z} and Lemma~\ref{1st log deriv zeta}, applied with $K=10^7$, we computed
$$\Re(v_{1,z}) \in [31.418062627034752, 31.418062627034846],$$
$$D(\mathcal{Z},z) \in [31.417963253430945, 31.417963255019071],$$
$$r(z,\tau,c)\in [0.000099372589781012325291744466523344471948495 \pm  5*10^{-45}].$$
Based on this input data, Theorem~\ref{main theorem zeta}, part(i), succeeded in verifying the RH 
for $\eta=224$, and Theorem~\ref{zeta counterpart}, part (i), succeeded
verifying the RH for $\eta=70216$, which is much larger and contains 
$1399910$ zeros of the zeta function. Theorem~\ref{zeta counterpart}, part (iii), also succeeded in 
verifying that completeness of 
the subset 
$$\mathcal{Z}\cap [y-\eta,y+\eta],\qquad \eta=49650,$$ 
a window that contains $989881$ zeros.
In the opposite direction, we applied Theorem~\ref{zeta counterpart}, part (iv), to 
the subset $\mathcal{Z}_0$, which is the same as $\mathcal{Z}$ except 
the subinterval $[\gamma_-,\gamma_+]$ corresponding to the ordinate
$\gamma = 10^{28} + 521738.816$
was removed. We computed
$$D(\mathcal{Z}_0,z) \in [31.417963247220145, 31.417963248808271],$$
$$R(z,\tau,c)\in [0.00009937291681087140202410471243137201884323 \pm 10^{-44}].$$
Based on this input data, 
Theorem~\ref{zeta counterpart} succeeded in proving that the set $\mathcal{Z}_0$ was indeed incomplete.

\subsection{Real Dirichlet $L$-function}
Let $d$ be a fundamental discriminant, 
$\chi_d$ be the corresponding real primitive character, and 
 $L(s,\chi_d)$ the corresponding Dirichlet $L$-function.
In the notation of \S{\ref{lfunction case}}, we have $r=1$, $N=d$, $\epsilon=1$, and 
if $d<0$ then $\mu_1=1$. 
We applied Theorem~\ref{main theorem l}, part (i),
to verify the RH using
$$d=-1159523,\qquad \delta = -1,\qquad \tau = 1692.8.$$ 
The coefficients arising from the Euler product are given by
$\alpha_{p,1}=\chi_d(p)$.
Our set $\mathcal{Z}$ was obtained using \cite{lcalc}. 
With the aid of Corollary~\ref{k=1 corollary l} 
as well as Lemma~\ref{1st log deriv 1} applied with $K=10^5$ 
we computed
$$w_{1,\delta} \in [6.4702225452, 6.4702573982],$$ 
$$C(\mathcal{Z},\delta) \in [6.4644405451, 6.4644405588].$$
Based on this input data, Theorem~\ref{main theorem l} succeeded in verifying the RH 
for $L(s,\chi_d)$ for $\eta=32$, a window containing $74$ zeros with nonnegative ordinates.

\subsection{The Ramanujan $\tau$ $L$-function}
Let $\tau$ be the Ramanujan tau function\footnote{So, $\tau(1) =1, \tau(2) = -24, \tau(3) = 252, \tau(4) = -1472,\ldots$.}, 
and let $L(s)$ be the Ramanujan tau $L$-function.\footnote{Therefore, 
$L(s)$ is given by the Dirichlet series
$L(s) = 1 +a_2 2^{-s}+ a_3 3^{-s}+a_4 4^{-s}+\cdots$ where $a_n=\tau(n) n^{-11/2}$,
at least when $\sigma > 1/2$.} 
In the notation of \S{\ref{lfunction case}}, we have $r=2$, $N=1$, $\epsilon=1$, $\mu_1=11/2$, and $\mu_2=13/2$.
We applied Theorem~\ref{main theorem l}
to verify the RH using
$$\delta = -1,\qquad \tau = 9877.3.$$ 
The coefficients $\alpha_{1,p}$ and $\alpha_{2,p}$ arising from the Euler product for $L(s)$ 
are given by the roots of the polynomial $x^2-\tau(p)p^{-11/2}x + 1$.  
Our set $\mathcal{Z}$ was obtained using \cite{lcalc} and \cite{lmfdb}. 
With the aid of Corollary~\ref{k=1 corollary l} 
and Lemma~\ref{1st log deriv 1}, applied with $K=10^5$, 
we computed
$$w_{1,\delta} \in [0.1671717623, 0.1672414682],$$ 
$$C(\mathcal{Z},\delta) \in [0.1663983945, 0.1663983946].$$
Based on this input data, Theorem~\ref{main theorem l} succeeded in verifying the RH for $L(s)$
for $\eta=84$, a window which includes $46$ zeros with nonnegative ordinates.

\subsection{Elliptic curve $L$-function}
Let $E$ be an elliptic curve over $\mathbb{Q}$ of conductor $N=\Delta_E$.
Let $L(s,E)$ be the corresponding elliptic curve $L$-function.
In the notation of \S{\ref{lfunction case}}, we have $r=2$, $N=\Delta_E$, $\epsilon= 1$ or $\epsilon=i$, 
$\mu_1=1/2$, and $\mu_2=3/2$.
We applied Theorem~\ref{main theorem l} with
$$\delta = -1,\qquad \tau = 90,$$ 
to verify the RH for the elliptic curve $E$ with 
minimal Weierstrass equation
\begin{equation}\label{E eq}
    E: y^2+y=x^3-x.
\end{equation}
According to \cite[37.a1]{lmfdb}, $E$ has conductor $37$ so that $N=37$, 
and the sign of the functional equation of $L(s,E)$ is $-1$ so that, 
in the notation of \S{\ref{lfunction case}}, $\epsilon= i$. 
To calculate the Euler factors of $L(s,E)$, let
$|E(\mathbb{F}_p)|$ denote 
the number of solutions $(x,y)\in \mathbb{F}_p\times \mathbb{F}_p$
that satisfy the minimal Weierstrass equation \eqref{E eq}
together with the point at infinity that lies on $E$. Define
\begin{equation}
b(p):=p+1-|E(\mathbb{F}_p)|.
\end{equation}
Then
the coefficients $\alpha_{1,p}$ and $\alpha_{2,p}$ arising from the Euler product for $L(s,E)$ 
are the roots of the polynomial $x^2-b(p)p^{-1/2}x + 1$, provided $p\ne 37$.
If $p=37$, then $\alpha_{p,1}=-1/\sqrt{p}$ and $\alpha_{p,2}=0$.
Our set $\mathcal{Z}$ was obtained from \cite{lmfdb}. 
With the aid of Corollary~\ref{k=1 corollary l} 
and Lemma~\ref{1st log deriv 1}, applied with $K=10^5$, 
we computed
$$w_{1,\delta} \in [1.2186382841, 1.21870798992],$$ 
$$C(\mathcal{Z},\delta) \in [1.160632197991927, 1.160632199964985].$$
Based on this input data, Theorem~\ref{main theorem l} succeeded in verifying the RH for $L(s,E)$
for $\eta=10$, a window which contains $5$ zeros with nonnegative ordinates.

\section{Proof of Proposition~\ref{zeros sum bound}}\label{proposition 12 proof}
Recall the function $\phi$ defined in \eqref{phi def}. 
For $\rho = \beta+i\gamma$, we have
$$\Re\, \left[\frac{1}{\rho-z}+ \frac{1}{1-\overline{\rho}-z}\right] = \phi(\beta,\gamma-y,x).$$
Also, if $|\gamma-y|>\tau>1-x$, then Lemma~\ref{phi lemma}, parts (ii--iii) give
\begin{equation}\label{phi bounds}
    \phi(0,\gamma-y,x)\le \phi(\beta,\gamma-y,x) \le \phi(1/2,\gamma-y,x).
\end{equation}
Now, for $u\ge 0$ let $N(u)$ be the number of zeta zeros with ordinates in $[0,u]$, and extend $N(u)$ to $u<0$ by requiring it to be odd. Using Stieltjes integrals, we define
\begin{align*}
    L(z,\tau)&:=\int_{|u-y|>\tau} \phi(0,u-y,x)\,\frac{dN(u)}{2},\\
    U(z,\tau)&:= \int_{|u-y|>\tau} \phi(1/2,u-y,x)\,\frac{dN(u)}{2}.
\end{align*}
The double inequality \eqref{phi bounds} thus gives
\begin{equation}\label{L U bounds}
    L(z,\tau)\le \sum_{\substack{\rho\\ |\gamma-y|>\tau}}\Re\, \frac{1}{\rho-z}
\le U(z,\tau).
\end{equation}
We bound $L$ and $U$ from below and above, respectively,  
starting with $L$. 

To this end, since the integrand in $L$ is nonnegative, a lower bound on $L$
can be obtained by restricting the integration interval to $\tau<|u-y|<y-c$. 
Doing so, followed by the change of variable $u\leftarrow u-y$, gives
\begin{equation}\label{L 1}
 L(z,\tau) \ge \frac{1}{2}\cdot \int_{\tau<u<y-c} \phi(0,u,x)\,d[N(y+u)-N(y-u)].    
\end{equation}
On the other hand, it is known \cite{davenport_1967} that 
\begin{equation}\label{N formula}
  N(u) = \frac{1}{\pi}\theta(u)+1+S(u),  
\end{equation}
where $\theta(u)= \arg [\pi^{-iu/2}\Gamma(1/4+iu/2)]$ and,
if $u$ does not coincide with the ordinate of 
any nontrivial zero, $S(u)=\pi^{-1}\arg \zeta(1/2+iu)$.\footnote{The arguments are defined by a continuous variation starting at $s=2$, going up vertically to $s=2+iu$, and then horizontally to $s=1/2+iu$.} Also, $\theta(u)$ is a smooth odd function and $S(u)$ is right-continuous with 
jump discontinuities at the zeros ordinates. Thus, combining \eqref{L 1} and \eqref{N formula}, and defining
\begin{align}
    L_{\theta}(z,\tau)&:=\int_{\tau<u<y-c} \phi(0,u,x)\,(\theta'(y+u)+\theta'(y-u))\,du,\label{L theta}\\ 
    L_S(z,\tau)&:=\int_{\tau<u<y-c} \phi(0,u,x)\,d[S(y+u)-S(y-u)],\label{L S}
\end{align}
where $\theta'$ is the derivative of $\theta$ with respect to $u$, 
we obtain
\begin{equation}\label{L 2}
L(z,\tau) \ge \frac{1}{2\pi} L_{\theta}(z,\tau)- \frac{1}{2} |L_S(z,\tau)|,
\end{equation}

We first bound $L_{\theta}$ from below.
By \cite[Lemma 10]{lehman_1970}, if $u>0$, then 
\begin{equation}\label{theta' ineq}
    \left|\theta'(u)- \frac{1}{2}\log \frac{u}{2\pi}\right|\le \frac{4\pi^2\cdot 0.006}{u^2}.
\end{equation}
So, in view of \eqref{L theta}, we are led to consider
\begin{equation}\label{log diff}
    \frac{1}{2}\log\frac{y+u}{2\pi}+\frac{1}{2}\log\frac{y-u}{2\pi}= \log\frac{y}{2\pi} +\frac{1}{2}\log \left(1-\frac{u^2}{y^2}\right).
\end{equation}
Expanding the right-side in \eqref{log diff} about $u=0$, and using the Lagrange form of the remainder, as well as \eqref{theta' ineq}, we see that for $\tau< u< y-c$, 
\begin{align}\label{theta' bound}
\left|\theta'(y+u)+\theta'(y-u) -\log\frac{y}{2\pi}\right|\le \frac{u^2}{y^2-u^2}+\epsilon_1(y,\tau,c),
\end{align}
where
\begin{equation*}
    \epsilon_1(y,\tau,c)=4\pi^2\cdot 0.006\cdot \left[\frac{1}{(y+\tau)^2}+\frac{1}{c^2}\right].
\end{equation*}

On the other hand, substituting the following simple bound into the integral in \eqref{phi ineq 2} below,
\begin{equation}\label{phi bound 3}
    0\le \phi(0,u,x) \le \frac{1-2x}{u^2},
\end{equation} 
and evaluating the resulting integral in closed-form gives the inequality
\begin{equation}\label{phi ineq 2}
    0\le \int_{\tau}^{y-c} \phi(0,u,x)\cdot \frac{u^2}{y^2-u^2}\,du \le \epsilon_2(z,c),
\end{equation}
where
\begin{equation*}
  \epsilon_2(z,c)=\frac{1-2x}{2y}\cdot \log\frac{2y}{c}.  
\end{equation*}
Additionally, using the anti-derivative formula
\begin{align}\label{1 integral}
     \int \phi(0,u,x)\,du = \arctan\left(\frac{u}{1-x}\right)-\arctan\left(\frac{u}{x}\right),
\end{align}
together with the following double inequality (from the Laurent series for $\arctan$), which is valid for $u>\tau>1-x$,
\begin{equation}\label{arctan bound 1}
0 \le    \left[\arctan\left(\frac{u}{1-x}\right)-\arctan\left(\frac{u}{x}\right)\right]-\left[\pi - \frac{1-2x}{u}\right]
\le \frac{(1-x)^3-x^3}{3u^3},
\end{equation}
we obtain
\begin{equation}\label{eps 2 bound}
\int_{\tau}^{y-c} \phi(0,u,x)\,du \ge \frac{1-2x}{\tau}- \epsilon_3(z,\tau,c),   
\end{equation}
where, after using the elementary inequality $(1-x)^3-x^3<(1-x)^2(1-2x)$,
\begin{equation*}
\epsilon_3(z,\tau,c)= \left[\frac{(1-x)^2}{3\tau^3}+\frac{1}{y-c}\right]\cdot (1-2x).    
\end{equation*}
Therefore, combining \eqref{L theta}, \eqref{theta' bound}, \eqref{phi ineq 2},
and \eqref{eps 2 bound}, we obtain
\begin{equation}\label{L theta bound}
    L_{\theta}(z,\tau) \ge \left(\frac{1-2x}{\tau}-\epsilon_3\right) \left(\log \frac{y}{2\pi}-\epsilon_1 \right)-\epsilon_2.
\end{equation}

We now calculate an upper bound on $L_S$, which is defined in \eqref{L S}. 
Let $\phi'$ denote the derivative of $\phi$ with respect to $u$. 
Using integration by parts and Lemma~\ref{phi lemma}, part (iv), together with the intermediate value theorem, we obtain
\begin{align*}
    \left|L_S(z,\tau)\right|&\le  2\,\left[\phi(0,\tau,x)+\phi(0,y-c,x)\right]\cdot \sup_{c<u<2y}|S(u)|\\
 &\quad -2\,\phi'(0,\tau,x)\cdot \sup_{c<u_1<u_2<2y}\left|\int_{u_1}^{u_2}S(u)\,du\right|.
\end{align*}
If $c> e$, then \cite[Theorem 1]{trudgian_2014} gives the bound $|S(u)| \le \ell(u)$. And 
if $c>168\pi$, then \cite[Theorem 2.2]{trudgian_2011} gives the bound $|\int_{u_0}^{u}S(t)\,t|\le \ell_1(u)$. 
So, using the simple bound \eqref{phi bound 3} for $\phi$ as well as the bound
\begin{equation}\label{phi' bound}
    \frac{-2+4x}{u^3}\le \phi'(\beta,u,x)\le 0,
\end{equation}
valid for $0\le \beta\le 1$, we obtain
\begin{equation}\label{L S bound}
    \left|L_S(z,\tau)\right|\le \epsilon_4(z,\tau,c) + \epsilon_5(z,\tau),
\end{equation}
where 
\begin{align*}
&\epsilon_4(z,\tau,c) = \left(\frac{2-4x}{\tau^2}+\frac{2-4x}{(y-c)^2}\right)\cdot \ell(2y), \\ 
&\epsilon_5(z,\tau) = \frac{4-8x}{\tau^3} \cdot \ell_1(2y).
\end{align*}
Combining \eqref{L 2}, \eqref{L theta bound}, \eqref{L S bound} yields
the lower bound in the proposition.

To derive the upper bound in the proposition, we bound $U(z,\tau)$ in \eqref{L U bounds} from above. 
Let us write $U(z,\tau) = I_1(z,\tau)+I_2(z,\tau)$ where
\begin{align*}
    I_1(z,\tau)&:= \int_{\tau<|u-y|\le y-c} \phi(1/2,u-y,x)\,\frac{dN(u)}{2},\\
    I_2(z,\tau) &:=  \int_{|u-y|> y-c} \phi(1/2,u-y,x)\,\frac{dN(u)}{2}.
\end{align*}
$I_1$ is estimated by an analogous calculation to that used for $L(z,\tau)$.
The difference is that the formula \eqref{1 integral} and the double inequality \eqref{arctan bound 1}  are replaced with
 the formula
\begin{equation*}
\int \phi(1/2,u,x)\,du = 2\arctan\left(\frac{2u}{1-2x}\right),
\end{equation*}
and the following double inequality, valid for $u>\tau>1-2x$,
\begin{equation*}
    0 \le    2\arctan\left(\frac{2u}{1-2x}\right)-\left[\pi - \frac{1-2x}{u}\right]
\le \frac{(1-2x)^3}{12u^3}.
\end{equation*}
Consequently,  the formula \eqref{eps 2 bound} is replaced with
\begin{equation*}
    \int_{\tau}^{y-c} \phi(1/2,u,x)\,du \le \frac{1-2x}{\tau}. 
\end{equation*}
So that the term $\epsilon_3$ in \eqref{L theta bound} may be replaced with zero. 
Also, since we are looking for an upper bound, the $-$ signs in \eqref{L theta bound} should be replaced 
with $+$ signs. Put together,
\begin{equation}\label{I1 bound}
    I_1(z,\tau) \le \frac{1-2x}{2\pi \tau} \left(\log \frac{y}{2\pi}+\epsilon_1 \right)+\frac{\epsilon_2}{2\pi}+\frac{\epsilon_4+\epsilon_5}{2}.
\end{equation}

Next, we bound $I_2$. After  the change of variable $u\leftarrow u-y$ we obtain
\begin{equation*}
    I_2(z,\tau) = \frac{1}{2}\int_{y-c}^{\infty} \phi(1/2,u,x)\,dN(y+u) - 
    \frac{1}{2}\int_{y-c}^{\infty} \phi(1/2,u,x)\,dN(y-u).
\end{equation*}
Using integration by parts together with the observation $\phi(1/2,u,x)\ll 1/u^2$ 
and the facts that $N(u)\ll u\log u$ and non-decreasing, we obtain
\begin{equation}\label{I2 bound}
    I_2(z,\tau) \le \int_{y-c}^{\infty} \left|\phi'(1/2,u,x)\right| N(y+u)\,du.
\end{equation}
Therefore, on substituting the bound on $\phi'$ given in \eqref{phi' bound} and the bound 
$$|N(u)|\le \frac{u}{2\pi}\log\frac{u}{2\pi},\qquad  (u\ge \gamma_1)$$ 
which follows from e.g.\ \cite[Corollary 1]{trudgian_2014}, we obtain after a small calculation
$$I_2(z,\tau)\le \epsilon_6(z,c),$$
where $\epsilon_6(z,c)$ is  defined as in the statement of the proposition. 
The claimed upper bound then follows on combining this with \eqref{I1 bound}.

\section{Proof of Lemma~\ref{phi lemma}}\label{lemma 6 proof}

\noindent
\textit{Proof of parts (i)--(iii):} Taking the partial derivative of $\phi$ with respect to $\beta$, we find with the aid of a computer algebra system that
\begin{equation}\label{phi beta}
    \frac{\partial}{\partial \beta} \phi(\beta,\eta,x) = 
    -\frac{(2\beta - 1)(2x-1)G(\beta,\eta,x)}{[((\beta-x)^2+\eta^2)((1-\beta-x)^2 + \eta^2)]^2},
\end{equation}
where $G$ is a degree $4$ monic polynomial in $\beta$,
\begin{align}\label{G poly}
     G(\beta,\eta,x) :=&\, \beta^4 - 2\beta^3 + (1 - 2\eta^2 + 2x - 2x^2)\beta^2 + 2(\eta^2 - x + x^2)\beta \\
     &+ \eta^2(2x - 2x^2 - 1 - 3\eta^2) + x^2(1 - 2x + x^2).\nonumber
\end{align}
satisfying $G(\beta,\eta,x)=G(1-\beta,\eta,x)$.
Note that the sign of the partial derivative of $\phi$ with respect to $\beta$ 
is the same as the sign of $(2\beta-1)G(\beta,\eta,x)$.

We have the formulas
\begin{align}
   G(\beta,0,x) &= (\beta-x)^2(1-\beta-x)^2, \label{0eta}\\
%   G(1/2,\eta,x)&=    \frac{1}{16}(1-2x)^4-\frac{1}{16}(1-2x)^4\eta^2-3\eta^4,\label{G beta1/2}\\
%    \frac{\partial}{\partial\eta}G(\beta,\eta,x) &= -\eta\left(12\eta^2 + (2\beta-1)^2 + (2x-1)^2\right),\label{deriv_eta}
 \frac{\partial}{\partial\beta} G(\beta,\eta,x) &= 2(2\beta-1) (\beta-r_+)(\beta-r_-),     \label{deriv_beta}
\end{align}
where 
\begin{equation} \label{rpm def}
    r_{\pm} := \frac{1 \pm \sqrt{4\eta^2 + (2x-1)^2}}{2} \qquad \text{so that}\qquad r_- < 0 \text{ and } 1 < r_+,
\end{equation}
as well as the formulas
\begin{align}
     \left[\frac{\partial^2}{\partial\beta^2} G(\beta,\eta,x) \right]_{\beta=\frac{1}{2}}  &=
    -4\eta^2 - (2x-1)^2,\label{2deriv_beta1/2}\\
    \left[\frac{\partial^2}{\partial\beta^2} G(\beta,\eta,x) \right]_{\beta=r_{\pm}}  &=
    8\eta^2 + 2 (2x-1)^2.\label{2deriv_betarpm}
\end{align}

    Taking \eqref{0eta} as our ``starting point" in some sense, and viewing it as a function of $\beta$, we see that
    $G(\beta,0,x)$ has two local minima (and $\beta$-axis intercepts) at $\beta = x$ and $\beta = 1-x$ with a (positive) local maximum at $\beta = 1/2$. As $\eta>0$ increases, we see from \eqref{deriv_beta} and \eqref{rpm def} together with \eqref{2deriv_betarpm}
     that the two local minima locations $r_{\pm}$ move away from $1/2$.
     In comparison, as follows from \eqref{deriv_beta} and \eqref{2deriv_beta1/2}, 
     $\beta = 1/2$ remains a local maximum of $G(\beta,\eta,x)$ (and a global maximum on the $\beta$-interval $[0,1]$),
     albeit with a monotonically decreasing value of $G(1/2,\eta,x)$. The latter claim can be seen from the negativity of the partial derivative
     \begin{equation}\label{eta deriv}
         \frac{\partial}{\partial\eta}G(\beta,\eta,x) = -\eta\left(12\eta^2 + (2\beta-1)^2 + (2x-1)^2\right).
     \end{equation}
     %We also note that by direct evaluation, the values $G(r_{\pm},\eta,x)$ are both negative.
     We now consider three cases.\\

    \noindent 
        \underline{Case (1)}: Suppose $G(\beta,\eta,x) > 0$ on the $\beta$-interval $(0,1)$, so that, by
        considering the sign of $\partial \phi/\partial \beta$ in \eqref{phi beta}, we find
        $\beta=1/2$ is a local minimum of $\phi(\beta,\eta,x)$. 
        In evaluating this case, we note by the negativity of the partial derivative in \eqref{eta deriv}, 
        if $\eta$ and $x$ are such that $\beta = 0$  is a root of $G(\beta,\eta,x)$, then $\beta=0$ cannot 
        be a root for any greater $\eta$ (with the same $x$). Also, $\beta=0$ is a root if and only if 
        the constant term of the polynomial 
        $G(\beta,\eta,x)$ in \eqref{G poly} is $0$, hence if and only if 
        $3\eta^4 + (2x^2-2x+1)\eta^2 + x^2(-x^2+2x-1)=0$. Solving for $\eta$ and 
        recalling the discussion following \eqref{2deriv_betarpm}, we therefore
        find $G(\beta,\eta,x) > 0$ throughout $\beta\in (0,1)$ if and only if 
        \[
            \eta \leq \sqrt{\frac{-2x^2+2x-1 + \sqrt{16x^4-32x^3+20x^2-4x+1}}{6}}=:v(x).
        \]
        For such $\eta$, $\phi(\beta,\eta,x)$ has no local extrema in the interval $\beta\in (0,1)$ 
        except at $\beta=1/2$ where it 
        has a local minimum. Thus, if $0\le \beta\le 1$, then $\phi(\beta,\eta,x)$ is minimized at $\beta=1/2$ in this case.\\

        \noindent
        \underline{Case (2)}: Suppose $G(\beta,\eta,x) < 0$ on the $\beta$-interval $(0,1)$, 
        so that, by considering the sign of $\partial \phi/\partial \beta$ in \eqref{phi beta}, we find
        $\beta=1/2$ is a local maximum of $\phi(\beta,\eta,x)$. 
        On the other hand, $G(\beta,\eta,x)$ maintains a negative sign throughout $\beta\in (0,1)$ if and only if 
        the local extremum of $G(\beta,\eta,x)$ that occurs at $\beta=1/2$ is negative. 
        By direct calculation, we have
        $$G(1/2,\eta,x)=    \frac{1}{16}(1-2x)^4-\frac{1}{16}(1-2x)^4\eta^2-3\eta^4.$$
        Setting $G(1/2,\eta,x) = 0$ and solving for $\eta$, we find that $G(\beta,\eta,x) < 0$ throughout $\beta\in (0,1)$ if and only if $$\eta > \frac{1-2x}{2\sqrt{3}}.$$
        %we consider when $\beta-1/2$ is a factor of $G(\beta,\eta,x)$. Since $G(\beta,\eta,x)$ is symmetric across $\beta=1/2$, this factor of %$\beta-1/2$ must have even multiplicity. In this case, we have the necessary equality
        %\[
        %\left(\beta-\frac12\right)^2(\beta^2-\beta + 4(-3n^4+n^2(-1+2x-2x^2)+x^2-2x^3+x^4)) = G(\beta,\eta,x).
        %\]
        %Comparing the linear (or equivalently the quadratic) coefficients yields
        %\[
        %12\eta^4 + 2(2x-1)^2\eta^2 - 4x^4 + 8x^3 - 6x^2 + 2x - \frac14 = 0.
        %\]
        %Solving for $\eta$, we find for a given $x$ that this is satisfied when
        %\[
        %\eta = \frac{1-2x}{2\sqrt{3}},
        %\]
        %and thus so long as $\eta > \frac{1-2x}{2\sqrt{3}}$, 
        For such $\eta$, $\phi(\beta,\eta,x)$ has no local extrema in the interval $\beta\in (0,1)$ except 
        at $\beta=1/2$ where it 
        has a local maximum. Thus, if $0\le \beta\le 1$, then 
        $\phi(\beta,\eta,x)$ is minimized at the boundary points $\beta=0,1$ and maximized at $\beta=1/2$ in this case.\\
        
        %then $G(\beta,\eta,x)$ is strictly negative for $\beta\in (0,1)$. In this case, $\phi(\beta,\eta,x)$ has a local maximum at $\beta = %\frac12$ and no other local extrema in the interval $\beta\in (0,1)$. Thus $\phi(\beta,\eta,x)$ is minimized at the boundary $\phi(0,\eta,x) %= \phi(1,\eta,x)$.\\

        \noindent
        \underline{Case (3)}: Suppose $G(\beta,\eta,x)$ does not maintain its sign throughtout the $\beta$-interval $(0,1)$, so that
        $G(\beta,\eta,x)$ has at least one root at some $\beta\in(0,1)$. By the symmetry of $G(\beta,\eta,x)$ about $\beta=1/2$
        as well as the discussion following \eqref{2deriv_betarpm}, there can be at most two such roots, one
        in the subinterval $(0,1/2]$ and another symmetric root in the subinterval $[1/2,1)$.
        By the work done thus far, such roots occur if and only if
        \[
        v(x) < \eta \le \frac{1-2x}{2\sqrt{3}}.
        \]
        When these roots occur, they each correspond to local maxima of $\phi(\beta,\eta,x)$ as seen by considering the sign of $\partial\phi/\partial\beta$ in \eqref{phi beta}. Therefore in this case, $\phi(\beta,\eta,x)$ is minimized either at the boundary $\beta = 0$ (equivalently $\beta = 1$), or at the center $\beta =1/2$. To find the point of transition between these two situations, we set $\phi(0,\eta,x) = \phi\left(1/2,\eta,x\right)$ and find that the transition 
        occurs when
        \[
        \eta = \sqrt{\frac{x(x-1)}{3}}.
        \]
    
    \vspace{2mm}    
     \noindent
        \underline{Summary}: From the above 3 cases, we have the following behavior of the minimum of 
        $\phi(\beta,\eta,x)$ for $\beta\in [0,1]$.
    \begin{enumerate}
        \item[(a)] If $\eta \le \sqrt{\frac{x(x-1)}{3}}$, then $\phi(\beta,\eta,x)$ is minimized at $\beta=\frac12$
        \item[(b)] If $\eta > \sqrt{\frac{x(x-1)}{3}}$, then $\phi(\beta,\eta,x)$ is minimized at $\beta=0$ (equivalently $\beta=1$).
    \end{enumerate}
    This covers claims (i) and (ii) in the statement of the lemma. Claim (iii) in the lemma 
    follows from Case (2) above.\\

\noindent
    \textit{Proof of part (iv):}
    Writing $\phi'(\beta,u,x)$ as the partial derivative of $\phi(\beta,u,x)$ with respect to $u$ and $\phi''(\beta,u,x)$ similarly. 
    We have
    \begin{align}
    \phi'(\beta,u,x) &= \frac{2(x-\beta)u}{[(x-\beta)^2 + u^2]^2} + \frac{2(x-(1-\beta))u}{[(x-(1-\beta))^2 + u^2]^2}\label{phi-1du}.
    %\phi''(\beta,u,x) &= \frac{2(\beta-x)[3u^2 - (\beta-x)^2]}{[(\beta-x)^2 + u^2]^3} + \frac{2(1-\beta-x)[3u^2 - (1-\beta-x)^2]}{[(1-\beta-x)^2 + u^2]^3}\label{phi-2du}
    \end{align}
    Since $x\le 0$, $\beta\in [0,1]$, and $u > 0$, by \eqref{phi-1du} we have that $\phi'(\beta,u,x) < 0$ for any such $\beta,u$, and $x$.
    
    For the claims regarding the increasing nature of $\phi'(\beta,u,x)$, we begin by evaluating $\phi''$ at $\beta = 1/2$ and find
    \[
    \phi''(1/2,u,x) = \frac{2(1-2x)[3u^2 - (\frac12 - x)^2]}{[(\frac12-x)^2 + u^2]^3}.
    \]
    From this, we see that $\phi'(1/2,u,x)$ is increasing if and only if 
    \[
    u > \frac{\frac12 - x}{\sqrt3} = \frac{1-2x}{2\sqrt3},
    \]
    yielding the first of these claims.

    Similarly, evaluating $\phi''$ at $\beta=0$ gives us
    \begin{equation}\label{phi-0b-2d}
        \phi''(0,u,x) = \frac{2x(x^2-3u^2)}{(x^2+u^2)^3} + \frac{2(x-1)((x-1)^2-3u^2)}{((x-1)^2+u^2)^3}.
    \end{equation}
    The first term in \eqref{phi-0b-2d} is positive only when $u > |x|/\sqrt3$ and the second term only when $u > |x-1|/\sqrt3$. 
    So, since $x\le 0$, $\phi'(0,u,x)$ is increasing when 
    \begin{equation*}\label{elem u cond}
      u > \frac{1-x}{\sqrt3} = \frac{2-2x}{2\sqrt3}.  
    \end{equation*}
    
%    \begin{remark}
%    If improvement of the elementary bound \eqref{elem u cond} is needed, one can explicitly determine the desired root of $\phi''(0,u,x)$ for 
%    $$u\in \left(\frac{|x|}{\sqrt{3}}, \frac{|x-1|}{\sqrt3}\right)$$ 
%    by setting the formula \eqref{phi-0b-2d} equal to $0$ and solving the resulting degree $4$ polynomial in the variable $u^2$.
%    \end{remark}

\section{Conclusions and future directions}
We presented a method to verify the RH for zeta at large heights and to verify the RH for a general class of $L$-functions at low heights. The method is simple to understand and implement and we demonstrated its efficacy on a variety of $L$-functions using interval arithmetic. We also presented a significant improvement to the method in the case of zeta by incorporating explicit bounds on $S(t)$ and integrals of $S(t)$.

In forthcoming work, we will develop and detail further generalizations of this verification method. These generalizations include, among other things, consideration of $w_{k,\delta}$ and $v_{k,z}$ when  $k>1$, further improvements in the case of zeta at large heights, the special case of Dirichlet $L$-functions to real primitive characters, as well as the extending of the improvements in \S{\ref{improvements}} to 
a more general setting.

\section*{Acknowledgements}
We are grateful to Georg-August-Universit\"at G\"ottingen,
Nieders\"achsische Staats- und Universit\"atsbibliothek G\"ottingen, for helping us locate some of
Riemann's unpublished notes and giving us permission to reproduce them in this paper. 
Megan Kyi thanks the OH5-OSU SURE Undergraduate Research program at the Ohio State University, Columbus, for 
their support.

%\cite{riemann3},\cite{riemann_1859},\cite{edwards_riemann_2001},\cite{clay},\cite{bailaud_bourget_1905},\cite{siegel_1932},\cite{trudgian_2014},\cite{trudgian_2016},\cite{lehmer_1988},\cite{lehman_1966},
%\cite{gram_1903},\cite{gram_1895},
%\cite{lindelof_1903}, \cite{backlund_1916},\cite{backlund_1918},\cite{backlund_1911},\cite{turing_1953},\cite{flint},\cite{li_1997}, \cite{lehman_1970}

\printbibliography
\end{document}